\documentclass[a4paper]{article}
\usepackage{fullpage}
\usepackage{authblk}
\usepackage{graphicx}
\usepackage[labelfont=bf]{caption}
\usepackage{placeins}
\usepackage{amsmath}
\usepackage{amsthm}
\usepackage{amsfonts}
\usepackage{amssymb}
\usepackage{bm}
\usepackage{hyperref}
\usepackage{xcolor}

\newtheorem{theorem}{Theorem}[section]
\newtheorem{definition}[theorem]{Definition}
\newtheorem{lemma}[theorem]{Lemma}
\newtheorem{corollary}[theorem]{Corollary}
\newtheorem{remark}[theorem]{Remark}
\newcommand{\dif}{\mathrm{d}}

\title{Augmented Gaussian Random Field: Theory and Computation}
\author[1]{Sheng Zhang}
\author[2]{Xiu Yang\thanks{xiy518@lehigh.edu}}
\author[1]{Samy Tindel}
\author[1,3]{Guang Lin\thanks{guanglin@purdue.edu}}
\date{\vspace{-5ex}}

\affil[1]{Department of Mathematics, Purdue University, West Lafayette, IN 47907, USA}
\affil[2]{Department of Industrial and Systems Engineering, Lehigh University, Bethlehem, PA 18015, USA}
\affil[3]{School of Mechanical Engineering, Purdue University, West Lafayette, IN 47907, USA}

\begin{document}
\maketitle

\noindent{}\textbf{Abstract:} We propose the novel augmented Gaussian random field (AGRF), which is a universal framework incorporating the data of observable and derivatives of any order. Rigorous theory is established. We prove that under certain conditions, the observable and its derivatives of any order are governed by a single Gaussian random field, which is the aforementioned AGRF. As a corollary, the statement ``the derivative of a Gaussian process remains a Gaussian process'' is validated, since the derivative is represented by a part of the AGRF. Moreover, a computational method corresponding to the universal AGRF framework is constructed. Both noiseless and noisy scenarios are considered. Formulas of the posterior distributions are deduced in a nice closed form. A significant advantage of our computational method is that the universal AGRF framework provides a natural way to incorporate arbitrary order derivatives and deal with missing data. We use four numerical examples to demonstrate the effectiveness of the computational method. The numerical examples are composite function, damped harmonic oscillator, Korteweg-De Vries equation, and Burgers' equation.

\noindent{}

\noindent{}\textbf{Keywords:} Gaussian random field, Gaussian process regression, arbitrary order derivatives, missing data, noisy data

\section{Introduction}
Gaussian random field (GRF) has been widely used in scientific and engineering study to construct a surrogate model (also called response surface or metamodel in different areas) of a complex system's observable based on available observations. Especially, its special case Gaussian process (GP) has become a powerful tool in applied math, statistics, machine learning, etc.~\cite{rasmussen_gaussian_2006}. Although random processes originally refer to one-dimensional random fields~\cite{abrahamsen_review_1997}, e.g., models describing time dependent systems, the terminology GP is interchangeable with GRF now in most application scenarios that involve high-dimensional systems. Also, in different areas, GRF-based (or GP-based) methods have different names. For example, in geostatistics, GP regression is referred to as Kriging, and it has multiple variants~\cite{sacks_design_1989, kitanidis_introduction_1997}.

To enhance the prediction accuracy of the GRF-based surrogate model, one can incorporate all the additional information available, such as gradients, Hessian, multi-fidelity data, physical laws, and empirical knowledge. For example, gradient-enhanced Kriging (GEK) uses gradients in either direct~\cite{morris_bayesian_1993} or indirect~\cite{chung_using_2002} way to improve the accuracy; multi-fidelity Cokriging combines a small amount of
high-fidelity data with a large amount of low-fidelity data from simplified or reduced models, in order to leverage low-fidelity models for speedup, while using high-fidelity data to establish accuracy and convergence guarantees~\cite{kennedy_predicting_2000, forrester_multi-fidelity_2007, le_gratiet_recursive_2014, perdikaris_nonlinear_2017}; physics-informed Kriging takes advantage of well-developed simulation tools to incorporate physical laws in the resulting GRF~\cite{yang_physics-informed_2018, yang_physics-informed_2019, yang_when_2020}; GP-based numerical ordinary/partial differential equation solver intrinsically imposes the equations in the structure of the GP, and it is one of the most important tools in probabilistic scientific computing~\cite{schober_probabilistic_2014, hennig_probabilistic_2015, chkrebtii_bayesian_2016, raissi_machine_2017, raissi_numerical_2018, cockayne_bayesian_2019}; inequality constraints that are not explicitly reflected in governing equations, e.g., positivity, monotonicity, can also be imposed to enhance accuracy and reduce uncertainty in the prediction~\cite{da_veiga_gaussian_2012, lin_bayesian_2014, lopez-lopera_finite-dimensional_2018, pensoneault_nonnegativity-enforced_2020}.

Despite the success in applying the aforementioned GRF-based methods to the construction of surrogate models for practical problems, the theory related to the accuracy, convergence, uncertainty, etc. of these methods is not well developed in general. Especially, when using the observable and its derivatives (such as gradient) jointly, e.g., to solve differential equations, one usually assumes that the random field representing the derivative is also a GRF, and its mean and covariance functions can be calculated by taking derivatives of the mean and covariance functions of the GRF representing the observable. Also, the correlation between the observable and its derivative is calculated accordingly \cite{rasmussen_gaussian_2006, adler_geometry_2010}. Intuitively, this is correct, because a linear combination of multiple Gaussian random variables is still a Gaussian random variable. However, taking limit of a linear combination is a critical step in the definition of derivative, which makes the theoretical guarantee of using derivatives in GRF-based methods less obvious. To the best of our knowledge, there is no comprehensive theory on the validity of the aforementioned assumption on the random field representing the derivative nor on its relation with the observable. Most literature uses the linearity of the derivative operator to validate this assumption (e.g.,~\cite{solak_derivative_2003}), which is questionable.

In this work, we propose the novel augmented Gaussian random field (AGRF), which is a universal framework incorporating the observable and its derivatives of any order. Rigorous theory is established. Under certain conditions, the observable and its derivatives of any order are governed by a single GRF, which is the aforementioned AGRF. As a corollary, the assumption ``the derivative of a Gaussian process remains a Gaussian process'' in~\cite{solak_derivative_2003} is validated, since the derivative is represented by a part of the AGRF. Furthermore, we construct a computational method corresponding to the universal AGRF framework. Both noiseless and noisy scenarios are considered. Formulas of the posterior distributions are deduced in a nice closed form. A significant advantage of our computational method is that the universal AGRF framework provides a natural way to incorporate arbitrary order derivatives and deal with missing data, e.g., the observation of the observable or its derivative is missing at some sampling locations.

This paper is organized as follows. The theoretical framework is presented in Section~\ref{sec:theo}, the computational framework is given in Section~\ref{sec:comp}, the numerical examples are provided in Section~\ref{sec:results}, and the conclusion follows in Section~\ref{sec:concl}.

\section{Theoretical framework}
\label{sec:theo}
In this section, we briefly review some fundamental definitions and an important theorem for GRF. Then we present our main theoretical framework, i.e., the AGRF, that unifies the GRFs representing the observable and its derivative. Finally, we extend it to the general case to incorporate information of derivatives of arbitrary order. For notation brevity, we consider the system's observable as a univariate function of the physical location or time. The extension to the multivariate case is straightforward. Therefore, our theoretical framework is applicable to GP methods that use the information of gradient as well as arbitrary order derivatives of the observable.

\subsection{Basic concepts}
In this paper, the Euclidean space $\mathbb{R}$ refers to $\mathbb{R}$ equipped with Euclidean metric and Lebesgue measure on Lebesgue-measurable sets. We begin with the definition of random fields.

\begin{definition}
Let $(\Omega, \mathcal{F}, \mathrm{P})$ be a probability space and $\mathbb{X}$ be a set. Suppose $f: \Omega \times \mathbb{X} \to \mathbb{R}$ is a mapping such that for each $x\in \mathbb{X}$, $f(\cdot, x): \Omega \to \mathbb{R}$ is a random variable (or measurable function). Then $f$ is called a real-valued random field on $\mathbb{X}$.
\end{definition}

We note that, in practical problems, $\mathbb{X}$ is typically a subset of the $d$-dimensional Euclidean space $\mathbb{R}^d$, i.e., $\mathbb{X}\subseteq \mathbb{R}^d$. Here, $\mathbb{X}$ is considered as a general set as
in~\cite{dudley_real_2002}. Next, we define Gaussian random fields as follows:

\begin{definition}
  Suppose $f$ is a real-valued random field on $\mathbb{X}$ such that for every finite set of indices $x_1, \cdots, x_p \in \mathbb{X}$, $(f(\omega, x_1), \cdots, f(\omega, x_p))$ is a multivariate Gaussian random variable, or, equivalently, every linear combination of $f(\omega, x_1), \cdots, f(\omega, x_p)$ has a univariate Gaussian (or normal) distribution. Then $f$ is called a Gaussian random field on $\mathbb{X}$.
\end{definition}

Here $\omega$ is an element in the sample space $\Omega$. A Gaussian random field is characterized by its mean and covariance functions:

\begin{definition}
  Given a Gaussian random field $f$ on $\mathbb{X}$, its mean function is defined as the expectation:
  \begin{equation}
    m(x) = \mathrm{E}[f(\omega, x)],
  \end{equation}
  and its covariance function is defined as:
  \begin{equation}
    k(x,x') = \mathrm{E}[(f(\omega, x) - m(x))(f(\omega, x') - m(x'))].
  \end{equation}
\end{definition}

Here, the covariance function is also called the kernel function. On the other hand, the following theorem provides a perspective in the converse way:

\begin{theorem}[Kolmogorov consistency theorem]\label{theorem1}
  Let $\mathbb{X}$ be any set, $m:\mathbb{X} \to \mathbb{R}$ any function, and $k: \mathbb{X} \times \mathbb{X} \to \mathbb{R}$ such that $k$ is symmetric and nonnegative definite. Then there exists a Gaussian random field on $\mathbb{X}$ with mean function $m$ and covariance function $k$.
\end{theorem}

\begin{proof}
  See~\cite{koralov_theory_2007} (p. 176) or~\cite{dudley_real_2002} (p. 443).
\end{proof}

\subsection{Main results}
We start our theoretical results with a limit in a Gaussian random field related to the derivative of its realization. The following theorem is a classic result with a proof in~\cite{cramer_stationary_2004} (p. 84). Here, we adapt the statement into our framework and reorganize the proof.

\begin{theorem}\label{theorem2}
  Let $f$ be a Gaussian random field on $\mathbb{R}$ with mean function $m(x)$ and covariance function $k(x,x')$ such that $m(x)$ is differentiable and $k(x,x')$ is twice continuously differentiable. Then there exists a real-valued random field $Df$ on $\mathbb{R}$ such that for each fixed $x\in \mathbb{R}$,
  \begin{equation}\label{eq:Df}
    \frac{f(\omega, x+\delta)-f(\omega, x)}{\delta} \xrightarrow[\delta\to 0]{\text{m.s.}} Df(\omega, x).
  \end{equation}
\end{theorem}

\begin{remark}
  Here ``m.s.'' stands for mean-square convergence, i.e., convergence in $L^2$. Since mean-square convergence implies convergence in probability and convergence in distribution, the limit in Theorem~\ref{theorem2} also holds in probability and in distribution.
\end{remark}

\begin{proof}[Proof of Theorem~\ref{theorem2}]
  We use the notations
  \begin{equation}
    m'(x) := \frac{\dif}{\dif x} m(x)
  \end{equation}
  and
  \begin{equation}
    k_{12}(x,x') := \frac{\partial^2}{\partial x \partial x'}k(x,x').
  \end{equation}
  Let $x\in \mathbb{R}$ be fixed. For any $\delta \in \mathbb{R}-\{0\}$, define
  \begin{equation}
    \xi_\delta = \xi_\delta(\omega) = \frac{f(\omega, x+\delta)-f(\omega, x)}{\delta}.
  \end{equation}
  Suppose $\delta, \tau \in \mathbb{R}-\{0\}$. Then
  \begin{equation}
    \begin{aligned}
      \mathrm{E}[\xi_\delta \xi_\tau] =&\; \mathrm{E}\left[\frac{(f(\omega, x+\delta)-f(\omega, x))(f(\omega, x+\tau)-f(\omega, x))}{\delta \tau}\right]  \\
      =&\; \frac{1}{\delta \tau}\{\mathrm{E}[f(\omega, x+\delta)f(\omega, x+\tau)] - \mathrm{E}[f(\omega, x+\delta)f(\omega, x)]  \\
      &\; - \mathrm{E}[f(\omega, x)f(\omega, x+\tau)] + \mathrm{E}[f(\omega, x)f(\omega, x)]\}  \\
      =&\; \frac{1}{\delta \tau}\{k(x+\delta, x+\tau) + m(x+\delta)m(x+\tau) -  k(x+\delta, x) - m(x+\delta)m(x)  \\
      &\; - k(x, x+\tau) - m(x)m(x+\tau) + k(x, x) + m(x)m(x)\}  \\
      =&\; \frac{1}{\delta \tau}\{k(x+\delta, x+\tau) -  k(x+\delta, x) - k(x, x+\tau) + k(x, x)\}   \\
      &\; + \frac{1}{\delta \tau}\{m(x+\delta)m(x+\tau) - m(x+\delta)m(x) - m(x)m(x+\tau) + m(x)m(x)\}.
    \end{aligned}
  \end{equation}
  Since $m(x)$ is differentiable and $k(x,x')$ is twice continuously differentiable,
  \begin{equation}
    \lim_{\delta, \tau \to 0}\mathrm{E}[\xi_\delta \xi_\tau] = k_{12}(x,x) + m'(x)m'(x).
  \end{equation}
  Therefore,
  \begin{equation}
    \lim_{\delta, \tau \to 0}\mathrm{E}[|\xi_\delta- \xi_\tau|^2] = \lim_{\delta, \tau \to 0}\{\mathrm{E}[\xi_\delta \xi_\delta] + \mathrm{E}[\xi_\tau\xi_\tau] - 2 \mathrm{E}[\xi_\delta \xi_\tau]\} = 0.
  \end{equation}
  Choose a sequence $\tau_i \to 0$ ($i = 1, 2, \cdots$) such that
  \begin{equation}
    \mathrm{E}[|\xi_{\tau_{i+1}}- \xi_{\tau_i}|^2] \le \frac{1}{2^{2i}}.
  \end{equation}
  Then
  \begin{equation}
    \mathrm{E}|\xi_{\tau_{i+1}}- \xi_{\tau_i}| \le \sqrt{\mathrm{E}[|\xi_{\tau_{i+1}}- \xi_{\tau_i}|^2]} \le \frac{1}{2^i}.
  \end{equation}
  By monotone convergence theorem,
  \begin{equation}
    \mathrm{E}\left[\sum_{i=1}^{\infty}|\xi_{\tau_{i+1}} - \xi_{\tau_i}|\right] = \sum_{i=1}^{\infty}\mathrm{E}|\xi_{\tau_{i+1}} - \xi_{\tau_i}| \le 1 < \infty.
  \end{equation}
  Thus,
  \begin{equation}
    \mathrm{P}\left(\sum_{i=1}^{\infty}|\xi_{\tau_{i+1}} - \xi_{\tau_i}|< \infty\right) = 1.
  \end{equation}
  So the random variable $\eta = \xi_{\tau_1} + \sum_{i=1}^{\infty}(\xi_{\tau_{i+1}} - \xi_{\tau_i})$ is well defined, and it is a candidate to be the limit in Eq.~\eqref{eq:Df}. Now, by monotone convergence theorem and Cauchy-Schwarz inequality, for any $j\ge 1$,
  \begin{equation}
    \begin{aligned}
      \mathrm{E}[|\eta - \xi_{\tau_j}|^2] =&\; \mathrm{E}\left[\left|\sum_{i=j}^{\infty}(\xi_{\tau_{i+1}} - \xi_{\tau_i})\right|^2\right]
      \le \mathrm{E}\left[\left(\sum_{i=j}^{\infty}|\xi_{\tau_{i+1}} - \xi_{\tau_i}|\right)^2\right]   \\
      =&\; \mathrm{E}\left[\left(\sum_{i=j}^{\infty}|\xi_{\tau_{i+1}} - \xi_{\tau_i}|\right)\left(\sum_{i'=j}^{\infty}|\xi_{\tau_{i'+1}} - \xi_{\tau_{i'}}|\right)\right]   \\
      =&\; \mathrm{E}\left[\sum_{i=j}^{\infty}\sum_{i'=j}^{\infty}|\xi_{\tau_{i+1}} - \xi_{\tau_i}||\xi_{\tau_{i'+1}} - \xi_{\tau_{i'}}|\right]  \\
      =&\; \sum_{i=j}^{\infty}\sum_{i'=j}^{\infty}\mathrm{E}|\xi_{\tau_{i+1}} - \xi_{\tau_i}||\xi_{\tau_{i'+1}} - \xi_{\tau_{i'}}|  \\
      \le&\; \sum_{i=j}^{\infty}\sum_{i'=j}^{\infty}\sqrt{\mathrm{E}[|\xi_{\tau_{i+1}} - \xi_{\tau_i}|^2] \mathrm{E}[|\xi_{\tau_{i'+1}} - \xi_{\tau_{i'}}|^2]}   \\
      \le&\; \sum_{i=j}^{\infty}\sum_{i'=j}^{\infty}\sqrt{\frac{1}{2^{2i}}\frac{1}{2^{2i'}}} = \sum_{i=j}^{\infty}\frac{1}{2^i}\sum_{i'=j}^{\infty}\frac{1}{2^{i'}} = \left(\frac{1}{2^{j-1}}\right)^2.
    \end{aligned}
  \end{equation}
  Since
  \begin{equation}
    \mathrm{E}[|\eta - \xi_\delta|^2] \le \mathrm{E}[2|\eta - \xi_{\tau_j}|^2 + 2|\xi_{\tau_j} - \xi_\delta|^2] = 2\mathrm{E}[|\eta - \xi_{\tau_j}|^2] + 2\mathrm{E}[|\xi_{\tau_j} - \xi_\delta|^2],
  \end{equation}
  we have
  \begin{equation}
    \lim_{\delta \to 0} \mathrm{E}[|\eta - \xi_\delta|^2] = 0,
  \end{equation}
  or
  \begin{equation}
    \xi_\delta \xrightarrow[\delta\to 0]{\text{m.s.}} \eta.
  \end{equation}
  Let $Df(\omega, x) = \eta$, do the same for every $x$, and the proof is complete.
\end{proof}

\begin{remark}
  For $\delta \in \mathbb{R}-\{0\}$, let
  \begin{equation}
    f_\delta = f_\delta(\omega, x) = \frac{f(\omega, x+\delta)-f(\omega, x)}{\delta}
  \end{equation}
  be a random field on $\mathbb{R}$. One could also consider the convergence of the family $\{f_\delta\;|\;\delta\in \mathbb{R}-\{0\}\}$ to $Df$. We have refrained from getting into this type of consideration for the sake of conciseness.
\end{remark}

The next lemma indicates that the limit of a series of Gaussian random variables is still a Gaussian random variable.

\begin{lemma}\label{lemma}
  Let $(\Omega, \mathcal{F}, \mathrm{P})$ be a probability space and $\xi_\delta$ ($\delta \in \mathbb{R}$) be a family of random variables such that $\xi_\delta$ ($\delta \neq 0$) are Gaussian random variables with mean $\mu_\delta$ and variance $\sigma_\delta^2$, and
  \begin{equation}
    \xi_\delta \xrightarrow[\delta\to 0]{\text{m.s.}} \xi_0.
  \end{equation}
  Then $\xi_0$ has Gaussian distribution with mean $\mu_0$ and variance $\sigma_0^2$, and
  \begin{equation}
  \begin{aligned}
    \lim_{\delta \to 0} \mu_\delta &= \mu_0  , \\
    \lim_{\delta \to 0} \sigma_\delta^2 &= \sigma_0^2.
  \end{aligned}
  \end{equation}
\end{lemma}

\begin{proof}
  Suppose $\delta, \tau \in \mathbb{R}-\{0\}$. Since
  \begin{equation}
    \lim_{\delta \to 0} \mathrm{E}[|\xi_\delta - \xi_0|^2] = 0
  \end{equation}
  and
  \begin{equation}
    \mathrm{E}[|\xi_\delta - \xi_\tau|^2] \le 2\mathrm{E}[|\xi_\delta - \xi_0|^2] + 2\mathrm{E}[|\xi_\tau - \xi_0|^2],
  \end{equation}
  we have
  \begin{equation}\label{eq:lemma_cauchy}
    \lim_{\delta, \tau \to 0} \mathrm{E}[|\xi_\delta - \xi_\tau|^2] = 0.
  \end{equation}
  Thus,
  \begin{equation}
    |\mu_\delta - \mu_\tau| = |\mathrm{E}[\xi_\sigma-\xi_\tau]|
    \le \mathrm{E}|\xi_\sigma-\xi_\tau| \\
    \le \sqrt{\mathrm{E}[|\xi_\sigma-\xi_\tau|^2]},
  \end{equation}
  which indicates
  \begin{equation}
    \lim_{\delta, \tau \to 0} |\mu_\delta - \mu_\tau| = 0.
  \end{equation}
  Hence, there exists $\mu_0 \in \mathbb{R}$ such that
  \begin{equation}\label{eq:lemma_mu}
    \lim_{\delta \to 0} \mu_\delta = \mu_0.
  \end{equation}
  By Eq.~\eqref{eq:lemma_cauchy}, there exists $\Delta>0$ such that for any $\delta$ ($0<|\delta|\le \Delta$), we have
  \begin{equation}
    \mathrm{E}[|\xi_\delta - \xi_\Delta|^2] \le 1.
  \end{equation}
  So
  \begin{equation}
    \mathrm{E}[\xi_\delta^2] \le 2\mathrm{E}[|\xi_\delta - \xi_\Delta|^2] + 2\mathrm{E}[\xi_\Delta^2] \le 2 + 2\mathrm{E}[\xi_\Delta^2].
  \end{equation}
  Since $\xi_\Delta$ is Gaussian, $\mathrm{E}[\xi_\Delta^2]$ is finite. Now, for any $\delta$ ($0<|\delta|\le \Delta$) and $\tau$ ($0<|\tau|\le \Delta$), by Cauchy-Schwarz inequality,
  \begin{equation}\label{eq:lemma_cs}
    \begin{aligned}
      |\sigma_\delta^2 - \sigma_\tau^2| &= |(\mathrm{E}[\xi_\delta^2] - \mu_\delta^2) - (\mathrm{E}[\xi_\tau^2] - \mu_\tau^2)|  \\
      &\le  |\mathrm{E}[\xi_\delta^2] - \mathrm{E}[\xi_\tau^2]| + |\mu_\delta^2 - \mu_\tau^2| \\
      &\le \mathrm{E}|\xi_\delta^2 - \xi_\tau^2| + |\mu_\delta^2 - \mu_\tau^2| \\
      &= \mathrm{E}|\xi_\delta - \xi_\tau||\xi_\delta + \xi_\tau| + |\mu_\delta^2 - \mu_\tau^2| \\
      &\le \sqrt{\mathrm{E}[|\xi_\delta - \xi_\tau|^2]}\sqrt{\mathrm{E}[|\xi_\delta + \xi_\tau|^2]} + |\mu_\delta^2 - \mu_\tau^2|  \\
      &\le \sqrt{\mathrm{E}[|\xi_\delta - \xi_\tau|^2]}\sqrt{2\mathrm{E}[\xi_\delta^2] + 2\mathrm{E}[\xi_\tau^2]} + |\mu_\delta^2 - \mu_\tau^2|  \\
      &\le \sqrt{\mathrm{E}[|\xi_\delta - \xi_\tau|^2]}\sqrt{8 + 8\mathrm{E}[\xi_\Delta^2]} + |\mu_\delta^2 - \mu_\tau^2|.
    \end{aligned}
  \end{equation}
  By Eqs.~\eqref{eq:lemma_cauchy}, \eqref{eq:lemma_mu}, and \eqref{eq:lemma_cs}, we have
  \begin{equation}
    \lim_{\delta, \tau \to 0} |\sigma_\delta^2 - \sigma_\tau^2| = 0.
  \end{equation}
  Since $\sigma_\delta^2 \ge 0$ for all $\delta \in \mathbb{R}-\{0\}$, there exists $\sigma_0^2 \ge 0$ such that
  \begin{equation}
    \lim_{\delta \to 0} \sigma_\delta^2 = \sigma_0^2.
  \end{equation}
  If $\sigma_0^2 = 0$, then
  \begin{equation}
    \begin{aligned}
      \mathrm{E}|\xi_0 - \mu_0| &\le \mathrm{E}|\xi_0 - \xi_\delta| + \mathrm{E}|\xi_\delta - \mu_\delta| + \mathrm{E}|\mu_\delta - \mu_0|  \\
      &\le \sqrt{\mathrm{E}[|\xi_0 - \xi_\delta|^2]} + \sqrt{\mathrm{E}[|\xi_\delta - \mu_\delta|^2]} + |\mu_\delta - \mu_0|  \\
      &= \sqrt{\mathrm{E}[|\xi_0 - \xi_\delta|^2]} + \sqrt{\sigma_\delta^2} + |\mu_\delta - \mu_0|.
    \end{aligned}
  \end{equation}
  Let $\delta \to 0$ and we have
  \begin{equation}
    \mathrm{E}|\xi_0 - \mu_0| = 0,
  \end{equation}
  or
  \begin{equation}
    \mathrm{P}(\xi_0 = \mu_0) = 1.
  \end{equation}
  Therefore, $\xi_0$ has degenerate Gaussian distribution with mean $\mu_0$ and variance $0$. If $\sigma_0^2 > 0$, then there exists $\Delta' > 0$ such that for any $\delta$ ($0 < |\delta| \le \Delta'$), we have
  \begin{equation}
    \sigma_\delta^2 \ge \sigma_0^2 / 2 > 0.
  \end{equation}
  For all $\delta$ ($|\delta| \le \Delta'$), define
  \begin{equation}
    \eta_\delta = \frac{\xi_\delta - \mu_\delta}{\sigma_\delta}.
  \end{equation}
  Then when $0 < |\delta| \le \Delta'$, $\eta_\delta$ has the standard Gaussian distribution. Now,
  \begin{equation}
    \begin{aligned}
      \mathrm{E}[|\eta_\delta - \eta_0|^2] &= \mathrm{E}\left[\left|\frac{\xi_\delta - \mu_\delta}{\sigma_\delta} - \frac{\xi_0 - \mu_0}{\sigma_0}\right|^2\right]  \\
      &\le 2\mathrm{E}\left[\left|\frac{\xi_\delta}{\sigma_\delta} - \frac{\xi_0}{\sigma_0}\right|^2\right] + 2\mathrm{E}\left[\left|\frac{\mu_\delta}{\sigma_\delta} - \frac{\mu_0}{\sigma_0}\right|^2\right]  \\
      &\le 4\mathrm{E}\left[\left|\frac{\xi_\delta}{\sigma_\delta} - \frac{\xi_\delta}{\sigma_0}\right|^2\right] + 4\mathrm{E}\left[\left|\frac{\xi_\delta}{\sigma_0} - \frac{\xi_0}{\sigma_0}\right|^2\right] + 2\left|\frac{\mu_\delta}{\sigma_\delta} - \frac{\mu_0}{\sigma_0}\right|^2   \\
      &= 4\left(\frac{1}{\sigma_\delta} - \frac{1}{\sigma_0}\right)^2 \mathrm{E}[\xi_\delta^2] + \frac{4}{\sigma_0^2}\mathrm{E}[|\xi_\delta - \xi_0|^2] + 2\left|\frac{\mu_\delta}{\sigma_\delta} - \frac{\mu_0}{\sigma_0}\right|^2   \\
      &= 4\frac{(\sigma_0 - \sigma_\delta)^2}{\sigma_\delta^2 \sigma_0^2} (\sigma_\delta^2 + \mu_\delta^2) + \frac{4}{\sigma_0^2}\mathrm{E}[|\xi_\delta - \xi_0|^2] + 2\left|\frac{\mu_\delta}{\sigma_\delta} - \frac{\mu_0}{\sigma_0}\right|^2.
    \end{aligned}
  \end{equation}
  Thus,
  \begin{equation}
    \lim_{\delta \to 0} \mathrm{E}[|\eta_\delta - \eta_0|^2] = 0.
  \end{equation}
  Let
  \begin{equation}
    \Phi (z) = \frac{1}{\sqrt{2\pi}}\int_{-\infty}^{z}e^{-t^2/2}dt
  \end{equation}
  be the cumulative distribution function (CDF) of the standard Gaussian distribution. For any $z\in \mathbb{R}$ and any $\epsilon > 0$,
  \begin{equation}
    \mathrm{P}(\eta_0 \le z) \le \mathrm{P}(\eta_\delta \le z + \epsilon) + \mathrm{P}(|\eta_\delta - \eta_0| > \epsilon)
    \le \Phi (z + \epsilon) + \frac{\mathrm{E}[|\eta_\delta - \eta_0|^2]}{\epsilon^2}.
  \end{equation}
  Let $\delta \to 0$ and we have
  \begin{equation}
    \mathrm{P}(\eta_0 \le z) \le \Phi (z + \epsilon).
  \end{equation}
  Let $\epsilon \to 0$ and we have
  \begin{equation}\label{eq:lemma_le}
    \mathrm{P}(\eta_0 \le z) \le \Phi (z).
  \end{equation}
  On the other hand,
  \begin{equation}
    \begin{aligned}
      \mathrm{P}(\eta_0 > z) &\le \mathrm{P}(\eta_\delta > z - \epsilon) + \mathrm{P}(|\eta_\delta - \eta_0| > \epsilon)  \\
      &\le 1 - \Phi (z - \epsilon) + \frac{\mathrm{E}[|\eta_\delta - \eta_0|^2]}{\epsilon^2}.
    \end{aligned}
  \end{equation}
  Let $\delta \to 0$ and we have
  \begin{equation}
    \mathrm{P}(\eta_0 > z) \le 1 - \Phi (z - \epsilon).
  \end{equation}
  Let $\epsilon \to 0$ and we have
  \begin{equation}
    \mathrm{P}(\eta_0 > z) \le 1 - \Phi (z),
  \end{equation}
  or
  \begin{equation}\label{eq:lemma_ge}
    \mathrm{P}(\eta_0 \le z) \ge \Phi(z).
  \end{equation}
  Combining Eqs.~\eqref{eq:lemma_le} and \eqref{eq:lemma_ge}, we have
  \begin{equation}
    \mathrm{P}(\eta_0 \le z) = \Phi(z).
  \end{equation}
  Thus, $\eta_0$ has standard Gaussian distribution, which implies that $\xi_0$ has Gaussian distribution with mean $\mu_0$ and variance $\sigma_0^2$.
\end{proof}

Now we define an augmented random field consisting of a Gaussian random field $f$ and the $Df$ associated with it. Then we prove that this augmented random field is a Gaussian random field.

\begin{definition}
  Let $f$ be a Gaussian random field on $\mathbb{R}$ with mean function $m(x)$ and covariance function $k(x,x')$ such that $m(x)$ is differentiable and $k(x,x')$ is twice continuously differentiable. The real-valued random field $Df$ defined in Theorem~\ref{theorem2} is called the derivative random field of $f$. Define a real-valued random field on $\mathbb{R} \times \{0,1\}$:
  \begin{equation}
    \tilde{f}: \Omega \times \mathbb{R} \times \{0,1\} \rightarrow \mathbb{R}
  \end{equation}
  such that
  \begin{equation}
    \begin{aligned}
      \tilde{f}(\omega, x, 0) &= f(\omega, x)   \\
      \tilde{f}(\omega, x, 1) &= Df(\omega, x).
    \end{aligned}
  \end{equation}
  We call $\tilde{f}$ as the augmented random field of $f$.
\end{definition}

\begin{theorem}\label{theorem3}
  Let $f$ be a Gaussian random field on $\mathbb{R}$ with mean function $m(x)$ and covariance function $k(x,x')$ such that $m(x)$ is differentiable and $k(x,x')$ is twice continuously differentiable. Then the augmented random field $\tilde{f}$ of $f$ is a Gaussian random field on $\mathbb{R} \times \{0,1\}$.
\end{theorem}

\begin{proof}
  For any $p, q \in \mathbb{N}^+\cup \{0\}$ such that $p+q\ge 1$, any $x_1, \cdots, x_p, y_1, \cdots, y_q \in \mathbb{R}$, and any \\ $c_1, \cdots, c_p, d_1, \cdots, d_q \in \mathbb{R}$, we have the linear combination:
  \begin{equation}
    \begin{aligned}
      &\; \sum_{i=1}^{p} c_i \tilde{f}(\omega, x_i, 0) + \sum_{j=1}^{q} d_j
      \tilde{f}(\omega, y_j, 1) \\
      =&\; \sum_{i=1}^{p} c_i f(\omega, x_i) + \sum_{j=1}^{q} d_j Df(\omega, y_j)  \\
      =&\; \sum_{i=1}^{p} c_i f(\omega, x_i) + \sum_{j=1}^{q} d_j \lim_{\delta_j \to 0}\frac{f(\omega, y_j+\delta_j)-f(\omega, y_j)}{\delta_j}   \\
      =&\; \lim_{\delta_1 \to 0} \cdots \lim_{\delta_q \to 0}\left\{\sum_{i=1}^{p} c_i f(\omega, x_i) + \sum_{j=1}^{q} d_j \frac{f(\omega, y_j+\delta_j)-f(\omega, y_j)}{\delta_j}\right\},
    \end{aligned}
  \end{equation}
  where the limits are taken in mean-square sense. Since $f$ is a Gaussian random field,
  \begin{equation}
    \sum_{i=1}^{p} c_i f(\omega, x_i) + \sum_{j=1}^{q} d_j \frac{f(\omega, y_j+\delta_j)-f(\omega, y_j)}{\delta_j}
  \end{equation}
  has Gaussian distribution for any $\delta_1, \cdots, \delta_q \in \mathbb{R} - \{0\}$. By Lemma~\ref{lemma}, the limit has Gaussian distribution. As this holds for every linear combination, $\tilde{f}$ is a Gaussian random field.
\end{proof}

After proving the augmented Gaussian random field is well defined, we calculate its mean and covariance functions.

\begin{theorem}\label{theorem4}
  Let $f$ be a Gaussian random field on $\mathbb{R}$ with mean function $m(x)$ and covariance function $k(x,x')$ such that $m(x)$ is differentiable and $k(x,x')$ is twice continuously differentiable. Then the augmented random field $\tilde{f}$ of $f$ has mean function:
  \begin{equation}
    \tilde{m}: \mathbb{R} \times \{0,1\} \rightarrow \mathbb{R}
  \end{equation}
  such that
  \begin{equation}
    \begin{aligned}
      \tilde{m}(x,0) &= m(x),  \\
      \tilde{m}(x,1) &= \frac{\dif}{\dif x} m(x),
    \end{aligned}
  \end{equation}
  and covariance function:
  \begin{equation}
    \tilde{k}: \mathbb{R} \times \{0,1\} \times \mathbb{R} \times \{0,1\}
    \rightarrow \mathbb{R}
  \end{equation}
  such that
  \begin{equation}
    \begin{aligned}
      \tilde{k}(x,0,x',0) &= k(x,x'), \\
      \tilde{k}(x,0,x',1) &= \frac{\partial}{\partial x'} k(x,x'), \\
      \tilde{k}(x,1,x',0) &= \frac{\partial}{\partial x} k(x,x'),  \\
      \tilde{k}(x,1,x',1) &= \frac{\partial^2}{\partial x \partial x'} k(x,x').
    \end{aligned}
  \end{equation}
\end{theorem}

\begin{proof}
  We use the notation
  \begin{equation}
    m'(x) := \frac{\dif}{\dif x} m(x).
  \end{equation}
  By the definition of $\tilde{f}$, we have
  \begin{equation}
    \tilde{m}(x,0) = \mathrm{E}[\tilde{f}(\omega, x, 0)] = \mathrm{E}[f(\omega, x)] = m(x).
  \end{equation}
  By Theorem~\ref{theorem2} and Lemma~\ref{lemma}, we have
  \begin{equation}
    \begin{aligned}
      \tilde{m}(x,1) &= \mathrm{E}[\tilde{f}(\omega, x, 1)]
       = \mathrm{E}[Df(\omega, x)]  \\
      &= \mathrm{E}\left[\lim_{\delta \to 0} \frac{f(\omega, x+\delta)-f(\omega, x)}{\delta}\right]  \\
      &= \lim_{\delta \to 0} \mathrm{E}\left[\frac{f(\omega, x+\delta)-f(\omega, x)}{\delta}\right]  \\
      &= \lim_{\delta \to 0} \frac{m(x+\delta)-m(x)}{\delta}  \\
      &= m'(x).
    \end{aligned}
  \end{equation}
  Similarly, by definition,
  \begin{equation}
    \begin{aligned}
      \tilde{k}(x,0,x',0) &= \mathrm{E}[(\tilde{f}(\omega, x, 0) - \tilde{m}(x,0))(\tilde{f}(\omega, x', 0) - \tilde{m}(x',0))]  \\
      &= \mathrm{E}[(f(\omega, x) - m(x))(f(\omega, x') - m(x'))]  \\
      &= k(x,x').
    \end{aligned}
  \end{equation}
  Also,
  \begin{equation}
    \begin{aligned}
      &\; \tilde{k}(x,0,x',1) \\
      =&\; \mathrm{E}[\tilde{f}(\omega, x, 0)\tilde{f}(\omega, x', 1)] - \mathrm{E}[\tilde{f}(\omega, x, 0)] \mathrm{E}[\tilde{f}(\omega, x', 1)]  \\
      =&\; \mathrm{E}[f(\omega, x)Df(\omega, x')] - m(x)m'(x')  \\
      =&\; \mathrm{E}\left[f(\omega, x)\left(Df(\omega, x')-\frac{f(\omega,
      x'+\delta) - f(\omega, x')}{\delta}\right)\right] \\
      &\; + \mathrm{E}\left[f(\omega, x)\frac{f(\omega, x'+\delta)  - f(\omega, x')}{\delta}\right] - m(x)m'(x')  \\
      =&\; \mathrm{E}\left[f(\omega, x)\left(Df(\omega, x')-\frac{f(\omega, x'+\delta) - f(\omega, x')}{\delta}\right)\right]  \\
      &\; + \frac{k(x,x'+\delta) + m(x)m(x'+\delta) - k(x,x') - m(x)m(x')}{\delta} - m(x)m'(x')  \\
      =&\; \mathrm{E}\left[f(\omega, x)\left(Df(\omega, x')-\frac{f(\omega, x'+\delta) - f(\omega, x')}{\delta}\right)\right]  \\
      &\; + \frac{k(x,x'+\delta) - k(x,x')}{\delta} + m(x)\left(\frac{m(x'+\delta)- m(x')}{\delta} - m'(x')\right).
    \end{aligned}
  \end{equation}
  Since
  \begin{equation}
    \begin{aligned}
      &\; \left|\mathrm{E}\left[f(\omega, x)\left(Df(\omega, x')-\frac{f(\omega, x'+\delta) - f(\omega, x')}{\delta}\right)\right]\right| \\
      \le&\; \mathrm{E}\left|f(\omega, x)\left(Df(\omega, x')-\frac{f(\omega, x'+\delta) - f(\omega, x')}{\delta}\right)\right|  \\
      \le&\; \sqrt{\mathrm{E}[|f(\omega, x)|^2]}\sqrt{\mathrm{E}\left[\left|Df(\omega, x')-\frac{f(\omega, x'+\delta) - f(\omega, x')}{\delta}\right|^2\right]},
    \end{aligned}
  \end{equation}
  let $\delta \to 0$ and we have
  \begin{equation}
    \tilde{k}(x,0,x',1) = \frac{\partial}{\partial x'} k(x,x').
  \end{equation}
  Similarly,
  \begin{equation}
    \tilde{k}(x,1,x',0) = \frac{\partial}{\partial x} k(x,x').
  \end{equation}
  Finally,
  \begin{equation}
    \begin{aligned}
      &\; \tilde{k}(x,1,x',1) \\
      =&\; \mathrm{E}[Df(\omega, x)Df(\omega, x')] - m'(x)m'(x') \\
      =&\; \mathrm{E}\left[Df(\omega, x)\left(Df(\omega, x')-\frac{f(\omega,
      x'+\delta) - f(\omega, x')}{\delta}\right)\right] \\
      &\; + \mathrm{E}\left[Df(\omega, x)\frac{f(\omega, x'+\delta) - f(\omega, x')}{\delta}\right]
      - m'(x)m'(x')  \\
      =&\; \mathrm{E}\left[Df(\omega, x)\left(Df(\omega, x')-\frac{f(\omega,
      x'+\delta) - f(\omega, x')}{\delta}\right)\right] \\
      &\; + \frac{\mathrm{E}[Df(\omega, x)f(\omega, x'+\delta)] - \mathrm{E}[Df(\omega, x)f(\omega, x')]}{\delta} - m'(x)m'(x')  \\
      =&\; \mathrm{E}\left[Df(\omega, x)\left(Df(\omega, x')-\frac{f(\omega, x'+\delta) - f(\omega, x')}{\delta}\right)\right]  \\
      &\; + \dfrac{\frac{\partial}{\partial x}k(x,x'+\delta) + m'(x)m(x'+\delta) -
      \frac{\partial}{\partial x}k(x,x') - m'(x)m(x')}{\delta} - m'(x)m'(x')  \\
      =&\; \mathrm{E}\left[Df(\omega, x)\left(Df(\omega, x')-\frac{f(\omega, x'+\delta) - f(\omega, x')}{\delta}\right)\right]  \\
      &\; + \dfrac{\frac{\partial}{\partial x}k(x,x'+\delta) - \frac{\partial}{\partial x}k(x,x')}{\delta} + m'(x)\left(\frac{m(x'+\delta) - m(x')}{\delta} - m'(x')\right).
    \end{aligned}
  \end{equation}
  Let $\delta \to 0$ and use the fact that $k(x,x')$ is twice continuously differentiable,
  \begin{equation}
    \tilde{k}(x,1,x',1) = \frac{\partial^2}{\partial x \partial x'} k(x,x').
  \end{equation}
\end{proof}

\begin{corollary}\label{corollary1}
  Let $f$ be a Gaussian random field on $\mathbb{R}$ with mean function $m(x)$ and covariance function $k(x,x')$ such that $m(x)$ is differentiable and $k(x,x')$ is twice continuously differentiable. Then the derivative random field
  $Df$ of $f$ is a Gaussian random field on $\mathbb{R}$ with mean function $\dif m(x)/\dif x$ and covariance function $\partial^2 k(x,x')/\partial x \partial x'$.
\end{corollary}

\begin{proof}[Sketch of proof]
  This is a direct conclusion from Theorem~\ref{theorem2}, Theorem~\ref{theorem3}, and Theorem~\ref{theorem4}.
\end{proof}

\subsection{Extensions}
The definition of the aforementioned augmented Gaussian random field can be extended to more general cases involving higher order derivatives.

\begin{definition}
  Let $f$ be a Gaussian random field on $\mathbb{R}$ such that the derivative random field $Df$ is a Gaussian random field on $\mathbb{R}$ with differentiable mean function and twice continuously differentiable covariance function. By Corollary~\ref{corollary1}, the derivative random field $D(Df)$ of $Df$ is a Gaussian random field on $\mathbb{R}$. Define the second order derivative random field of $f$ as $D^2 f = D(Df)$. Recursively, define the $n^{\text{th}}$ order derivative random field of $f$ as $D^n f = D(D^{n-1} f)$ when $D^{n-1} f$ is a Gaussian random field on $\mathbb{R}$ with differentiable mean function and twice continuously differentiable covariance function.
\end{definition}

\begin{corollary}\label{corollary2}
  Let $f$ be a Gaussian random field on $\mathbb{R}$ with mean function $m(x)$ and covariance function $k(x,x')$ such that $m(x)$ is $n$ times differentiable and $k(x,x')$ is $2n$ times continuously differentiable ($n \in \mathbb{N}^+$). Then $Df, \cdots, D^n f$ are well-defined and are Gaussian random fields on $\mathbb{R}$.
\end{corollary}

\begin{proof}[Sketch of proof]
  This corollary can be proved by applying Corollary~\ref{corollary1} recursively.
\end{proof}

Now we can define the general augmented Gaussian random field involving higher order derivatives.

\begin{definition}
  Let $f$ be a Gaussian random field on $\mathbb{R}$ with mean function $m(x)$ and covariance function $k(x,x')$ such that $m(x)$ is $n$ times differentiable and $k(x,x')$ is $2n$ times continuously differentiable ($n \in \mathbb{N}^+$). Define the $n^{\text{th}}$ order augmented random field of $f$ as:
  \begin{equation}
    \tilde{f}^n: \Omega \times \mathbb{R} \times \{0,1,\cdots,n\}
    \rightarrow\mathbb{R}
  \end{equation}
  such that
  \begin{equation}
    \begin{aligned}
      \tilde{f}^n(\omega, x, 0) &= f(\omega, x)   \\
      \tilde{f}^n(\omega, x, 1) &= Df(\omega, x)   \\
      &\vdots      \\
      \tilde{f}^n(\omega, x, n) &= D^n f(\omega, x).
    \end{aligned}
  \end{equation}
\end{definition}

\begin{theorem}\label{theorem5}
  Let $f$ be a Gaussian random field on $\mathbb{R}$ with mean function $m(x)$ and covariance function $k(x,x')$ such that $m(x)$ is $n$ times differentiable and $k(x,x')$ is $2n$ times continuously differentiable ($n \in \mathbb{N}^+$). Then the $n^{\text{th}}$ order augmented random field $\tilde{f}^n$ of $f$ is a Gaussian random field on $\mathbb{R} \times \{0,1,\cdots,n\}$.
\end{theorem}

\begin{proof}[Sketch of proof]
  This can be proved in a similar way to the proof of Theorem~\ref{theorem3}.
\end{proof}

The following theorem calculates the mean and covariance functions of the $n^{\text{th}}$ order augmented Gaussian random field.

\begin{theorem}\label{theorem6}
  Let $f$ be a Gaussian random field on $\mathbb{R}$ with mean function $m(x)$ and covariance function $k(x,x')$ such that $m(x)$ is $n$ times differentiable and $k(x,x')$ is $2n$ times continuously differentiable ($n \in \mathbb{N}^+$). Then the $n^{\text{th}}$ order augmented random field $\tilde{f}^n$ of $f$ has mean function:
  \begin{equation}
    \begin{aligned}
      \tilde{m}^n : & ~\mathbb{R} \times \{0,1,\cdots,n\} \rightarrow \mathbb{R} \\
      & \tilde{m}^n (x,i) = \dfrac{\dif^i}{\dif x^i} m(x),
    \end{aligned}
  \end{equation}
  and covariance function:
  \begin{equation}
    \begin{aligned}
      \tilde{k}^n: & ~\mathbb{R} \times \{0,1,\cdots,n\} \times \mathbb{R} \times
      \{0,1,\cdots,n\} \rightarrow \mathbb{R}  \\
      & \tilde{k}^n (x,i,x',j) = \dfrac{\partial^{i+j}}{\partial x^i \partial x'^j} k(x,x').
    \end{aligned}
  \end{equation}
\end{theorem}

\begin{proof}[Sketch of proof]
  When $i,j \in \{0,1\}$, by Theorem~\ref{theorem4}, the formulas hold. Then this theorem can be proved by using induction and following a similar way to the proof of Theorem~\ref{theorem4}.
\end{proof}

Furthermore, we can extend the augmented Gaussian random field to the infinite order case, and calculate the mean and covariance functions accordingly.

\begin{definition}
  Let $f$ be a Gaussian random field on $\mathbb{R}$ with mean function $m(x)$ and covariance function $k(x,x')$ such that $m(x)$ and $k(x,x')$ are smooth. Define the infinite order augmented random field of $f$ as:
  \begin{equation}
    \tilde{f}^\infty: \Omega \times \mathbb{R} \times \mathbb{N}
    \rightarrow \mathbb{R}
  \end{equation}
  such that
  \begin{equation}
    \begin{aligned}
      \tilde{f}^\infty(\omega, x, 0) &= f(\omega, x)  \\
      \tilde{f}^\infty(\omega, x, 1) &= Df(\omega, x)  \\
      &\vdots      \\
      \tilde{f}^\infty(\omega, x, n) &= D^n f(\omega, x)  \\
      &\vdots
    \end{aligned}
  \end{equation}
\end{definition}

\begin{theorem}
  Let $f$ be a Gaussian random field on $\mathbb{R}$ with mean function $m(x)$ and covariance function $k(x,x')$ such that $m(x)$ and $k(x,x')$ are smooth. Then the infinite order augmented random field $\tilde{f}^\infty$ of $f$ is a Gaussian random field on $\mathbb{R} \times \mathbb{N}$.
\end{theorem}

\begin{proof}[Sketch of proof]
  This is proved in a similar way to the proof in Theorem~\ref{theorem3}.
\end{proof}

\begin{theorem}
Let $f$ be a Gaussian random field on $\mathbb{R}$ with mean function $m(x)$ and covariance function $k(x,x')$ such that $m(x)$ and $k(x,x')$ are smooth. Then the infinite order augmented random field $\tilde{f}^\infty$ of $f$ has mean function:
  \begin{equation}
    \begin{aligned}
      \tilde{m}^\infty & : ~\mathbb{R} \times \mathbb{N} \to \mathbb{R}, \\
      & \tilde{m}^\infty (x,i) = \dfrac{\dif^i}{\dif x^i} m(x),
    \end{aligned}
  \end{equation}
  and covariance function:
  \begin{equation}
    \begin{aligned}
      \tilde{k}^\infty & : ~\mathbb{R} \times \mathbb{N} \times \mathbb{R}
      \times \mathbb{N} \to \mathbb{R},  \\
      & \tilde{k}^\infty (x,i,x',j) = \dfrac{\partial^{i+j}}{\partial x^i \partial x'^j} k(x,x').
    \end{aligned}
  \end{equation}
\end{theorem}

\begin{proof}[Sketch of proof]
  When $i,j \in \{0,1\}$, by Theorem~\ref{theorem4}, the formulas hold. Then the result is proved by induction and in a similar way to Theorem~\ref{theorem4}.
\end{proof}

\section{Computational framework}
\label{sec:comp}
In this section, we describe the computational framework for the AGRF prediction. Noiseless and noisy scenarios are considered. Since we use univariate observable in the theoretical framework, we illustrate the same scenario here for consistency and conciseness. Formulations for the multi-variate cases can be deduced based on our results and the gradient-enhanced Kriging/Cokriging methods (see, e.g.,~\cite{ulaganathan_performance_2015, laurent_overview_2019, deng_multifidelity_2020}).

\subsection{Prediction using noiseless data}
As in the conventional GP regression, we aim to condition the joint Gaussian prior distribution on the observations, as such to provide a posterior joint Gaussian distribution. In our framework, the observations include the collected data of the observable and its derivatives of different orders. Suppose we are given a finite collection of real-valued data pairs:
\begin{equation}
  \begin{aligned}
    \text{Observable:}                    & \quad(x_{01}, y_{01}) \; (x_{02}, y_{02}) \; \cdots \; (x_{0p_0}, y_{0p_0}) \\
    \text{First order derivative:}        & \quad(x_{11}, y_{11}) \; (x_{12}, y_{12}) \; \cdots \; (x_{1p_1}, y_{1p_1}) \\
    \text{Second order derivative:}       & \quad(x_{21}, y_{21}) \; (x_{22}, y_{22}) \; \cdots \; (x_{2p_2}, y_{2p_2}) \\
    & \quad\vdots \\
    n^\text{th} \text{ order derivative:} & \quad(x_{n1}, y_{n1}) \; (x_{n2}, y_{n2}) \; \cdots \; (x_{np_n}, y_{np_n})
  \end{aligned}
\end{equation}
with $n \ge 0$ and $p_0, p_1, p_2, \cdots, p_n \ge 0$. Here, $x_{ij}$ are locations and $y_{ij}$ are the data collected at this location. Of note, we consider a general case, and it is not necessary that $x_{ij}=x_{(i+1)j}$. In other words, it is possible that the observable and its derivatives are sampled at different locations. We assume that a mean function and a covariance function are given for the Gaussian random field $f$ that describes the observable:
\begin{equation}
  m:\mathbb{R} \to \mathbb{R}
\end{equation}
and
\begin{equation}
  k: \mathbb{R} \times \mathbb{R} \to \mathbb{R}
\end{equation}
such that $m$ is $n$ times differentiable and $k$ is symmetric, nonnegative definite, and $2n$ times continuously differentiable. By Theorem~\ref{theorem1} and Theorem~\ref{theorem5}, there exists a Gaussian random field $\tilde{f}^n$ on $\mathbb{R} \times \{0,1,\cdots,n\}$ whose mean function and covariance function are given by Theorem~\ref{theorem6}. We use the augmented Gaussian random field $\tilde{f}^n$ to model the data such that
\begin{equation}
  \tilde{f}^n(\omega, x_{ij}, i) = y_{ij} \text{ for } i\in \{0,1,\cdots,n\}
  \text{ and } j\in \{1,\cdots,p_i\}.
\end{equation}
The prediction of the $q^\text{th}$ ($0\le q\le n$) order derivative at any $x_{q*}\in \mathbb{R}$ is the posterior mean of the random variable $\tilde{f}^n(\omega, x_{q*}, q)$ (denoted by $y_{q*}$), and the uncertainty in the prediction can be described by the confidence interval based on the posterior variance (or standard deviation).

We introduce the notations:
\begin{equation}
  \begin{aligned}
    \bm x_0 &= \left[x_{01}, x_{02}, \cdots, x_{0p_0}\right]^\text{T} \\
    \bm x_1 &= \left[x_{11}, x_{12}, \cdots, x_{1p_1}\right]^\text{T} \\
            &\quad\vdots \\
    \bm x_n &= \left[x_{n1}, x_{n2}, \cdots, x_{np_n}\right]^\text{T}
  \end{aligned}
\end{equation}
and
\begin{equation}
  \begin{aligned}
    \bm y_0 &= \left[y_{01}, y_{02}, \cdots, y_{0p_0}\right]^\text{T} \\
    \bm y_1 &= \left[y_{11}, y_{12}, \cdots, y_{1p_1}\right]^\text{T} \\
            &\quad\vdots \\
    \bm y_n &= \left[y_{n1}, y_{n2}, \cdots, y_{np_n}\right]^\text{T}.
  \end{aligned}
\end{equation}
Since $\tilde{f}^n$ is a Gaussian random field, we have the following multivariate Gaussian distribution:
\begin{equation}\label{eq:prior}
  \begin{bmatrix}
    y_{q*} \\
    \bm y_0 \\
    \vdots \\
    \bm y_n
  \end{bmatrix}
  \thicksim
  \mathcal{N}
  \left(\begin{bmatrix}
    M_* \\
    M_0 \\
    \vdots \\
    M_n
  \end{bmatrix},
  \begin{bmatrix}
    K_{**} & K_{*0} & \cdots & K_{*n} \\
    K_{0*} & K_{00} & \cdots & K_{0n} \\
    \vdots & \vdots & \ddots & \vdots \\
    K_{n*} & K_{n0} & \cdots & K_{nn}
  \end{bmatrix}
  \right),
\end{equation}
where
\begin{equation}
  \begin{aligned}
    M_* &= \tilde{m}^n(x_{q*}, q) \\
    M_i &= \left[\tilde{m}^n(x_{i1}, i), \cdots, \tilde{m}^n(x_{ip_i}, i)\right]^\text{T}
  \end{aligned}
\end{equation}
and
\begin{equation}
  \begin{aligned}
    K_{**} &= \tilde{k}^n(x_{q*},q,x_{q*},q) \\
    K_{*i'} &= \left[\tilde{k}^n(x_{q*},q,x_{i'1},i'), \cdots, \tilde{k}^n(x_{q*},q,x_{i'p_{i'}},i')\right] \\
    K_{i*} &= \left[\tilde{k}^n(x_{i1},i,x_{q*},q), \cdots, \tilde{k}^n(x_{ip_i},i,x_{q*},q)\right]^\text{T} \\
    K_{ii'} &= \begin{bmatrix}
      \tilde{k}^n(x_{i1},i,x_{i'1},i') & \cdots & \tilde{k}^n(x_{i1},i,x_{i'p_{i'}},i') \\
      \vdots & \ddots & \vdots \\
      \tilde{k}^n(x_{ip_i},i,x_{i'1},i') & \cdots & \tilde{k}^n(x_{ip_i},i,x_{i'p_{i'}},i')
    \end{bmatrix}
  \end{aligned}
\end{equation}
for $i,i'\in \{0,1,\cdots,n\}$. Here, $\tilde{m}^n$ and $\tilde{k}^n$ are calculated according to Theorem~\ref{theorem6}. Then, the posterior distribution of $y_{q*}$ is also a Gaussian distribution:
\begin{equation}\label{eq:posterior}
  \left(y_{q*}\; \big |\;\bm y_0, \cdots, \bm y_n \right) \thicksim \mathcal{N}(\mu, \sigma^2),
\end{equation}
where
\begin{equation}
  \begin{aligned}
    \mu &= M_* + \left[K_{*0}, \cdots, K_{*n}\right] \begin{bmatrix}
      K_{00} & \cdots & K_{0n} \\
      \vdots & \ddots & \vdots \\
      K_{n0} & \cdots & K_{nn}
    \end{bmatrix}^{-1}
    \begin{bmatrix}
      \bm y_0 - M_0 \\
      \vdots \\
      \bm y_n - M_n
    \end{bmatrix} \\
    \sigma^2 &= K_{**} - \left[K_{*0}, \cdots, K_{*n}\right]
    \begin{bmatrix}
      K_{00} & \cdots & K_{0n} \\
      \vdots & \ddots & \vdots \\
      K_{n0} & \cdots & K_{nn}
    \end{bmatrix}^{-1}
    \begin{bmatrix}
      K_{0*} \\
      \vdots \\
      K_{n*}
    \end{bmatrix}.
  \end{aligned}
\end{equation}

In practice, we usually assume the form of the mean function $m$ and the covariance function $k$, which involve hyperparameters. We denote these parameters by a vector $\bm \theta$. For instance, in the widely used squared exponential covariance function $k(x,x')=a^2\exp(-(x-x')^2 / (2 l^2))$, $a$ and $l$ are hyperparameters. Similarly, the mean function $m$ may include hyperparameters as well. For example, if $m$ is a polynomial as in the universal Kriging, the coefficients of the polynomial are hyperparameters. Similar to the standard GP method, the AGRF method identifies the hyperparameters via maximizing the following log-likelihood:
\begin{multline}
  \log p\left(\bm y_0, \cdots, \bm y_n \;\big |\; \bm x_0, \cdots, \bm x_n, \bm \theta\right) = -\frac{p_0+\cdots+p_n}{2}\log(2\pi) - \frac{1}{2} \log\left(\det\left(\begin{bmatrix}
    K_{00} & \cdots & K_{0n} \\
    \vdots & \ddots & \vdots \\
    K_{n0} & \cdots & K_{nn}
  \end{bmatrix}\right)\right) \\
  -\;\frac{1}{2} \begin{bmatrix}
    \bm y_0 - M_0 \\
    \vdots \\
    \bm y_n - M_n
  \end{bmatrix}^\text{T} \begin{bmatrix}
    K_{00} & \cdots & K_{0n} \\
    \vdots & \ddots & \vdots \\
    K_{n0} & \cdots & K_{nn}
  \end{bmatrix}^{-1} \begin{bmatrix}
    \bm y_0 - M_0 \\
    \vdots \\
    \bm y_n - M_n
  \end{bmatrix}.
\end{multline}
After identifying $\bm \theta$, we obtain the posterior distribution in Eq.~\eqref{eq:posterior}.

\subsection{Prediction using noisy data}
\begin{figure}[t]
  \centering
  \includegraphics[width=0.99\linewidth]{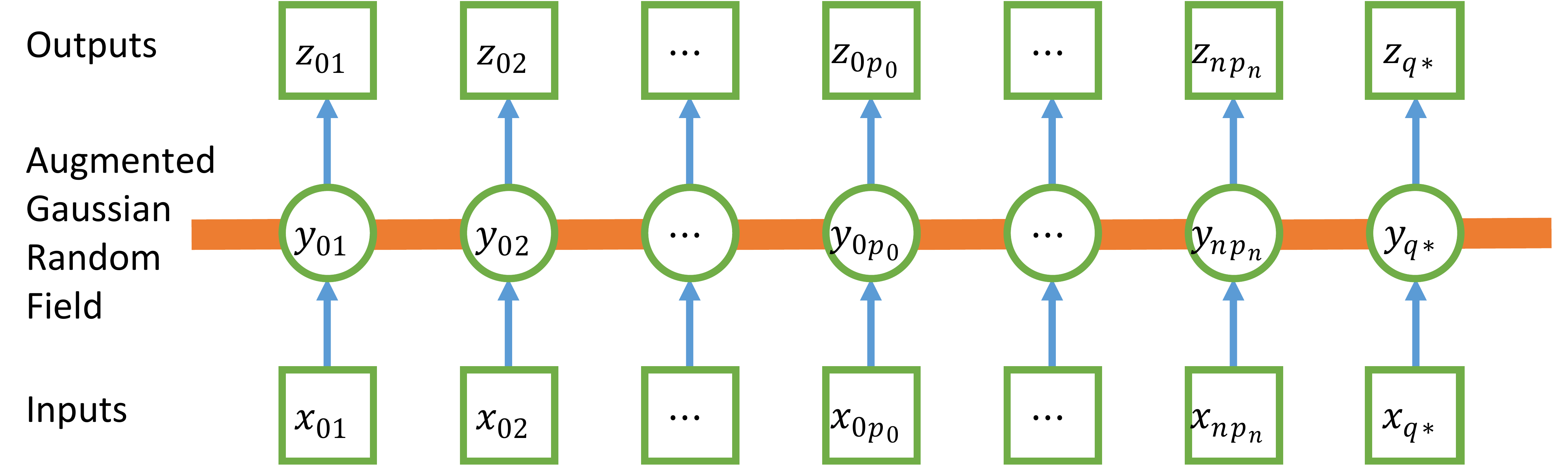}
  \caption{\textbf{Graphical illustration of augmented Gaussian random field prediction with measurement noise.} There are three layers: input layer, hidden layer, and output layer. The hidden layer is dominated by augmented Gaussian random field. The observable and its derivatives of different orders are integrated into the same field to make predictions.}
  \label{fig:diagram}
\end{figure}

Suppose we are given a finite collection of real-valued noisy data pairs:
\begin{equation}
  \begin{aligned}
    \text{Observable:}                    & \quad(x_{01}, z_{01}) \; (x_{02}, z_{02}) \; \cdots \; (x_{0p_0}, z_{0p_0}) \\
    \text{First order derivative:}        & \quad(x_{11}, z_{11}) \; (x_{12}, z_{12}) \; \cdots \; (x_{1p_1}, z_{1p_1}) \\
    \text{Second order derivative:}       & \quad(x_{21}, z_{21}) \; (x_{22}, z_{22}) \; \cdots \; (x_{2p_2}, z_{2p_2}) \\
    & \quad\vdots \\
    n^\text{th} \text{ order derivative:} & \quad(x_{n1}, z_{n1}) \; (x_{n2}, z_{n2}) \; \cdots \; (x_{np_n}, z_{np_n})
  \end{aligned}
\end{equation}
with $n \ge 0$ and $p_0, p_1, p_2, \cdots, p_n \ge 0$. Here $x_{ij}$ are locations and $z_{ij}$ are the noisy data collected at this location. Assume that $z_{ij}$ can be decomposed into the sum of noiseless part $y_{ij}$ and noise $\epsilon_{ij}$, where $y_{ij}$ are governed by AGRF and $\epsilon_{ij}$ are Gaussian random variables independent of each other and independent of all the $y_{ij}$:
\begin{equation}\label{eq:decomp}
  \begin{aligned}
    z_{ij} &= y_{ij} + \epsilon_{ij} \\
    \epsilon_{ij} &\sim \mathcal{N}(0,\delta_i^2),
  \end{aligned}
\end{equation}
where $\delta_i$ is the noise intensity of the $i^\text{th} \text{ order derivative}$. We use different noise intensities for different order derivatives because different order derivatives might have different magnitudes. Then, we have
\begin{equation}\label{eq:var}
  \begin{aligned}
    \text{Var}[z_{ij}] &= \text{Var}[y_{ij}] + \delta_i^2 \\
    \text{Cov}[z_{ij},z_{i'j'}] &= \text{Cov}[y_{ij},y_{i'j'}] \text{ for } (i,j)\neq(i',j').
  \end{aligned}
\end{equation}
See Figure~\ref{fig:diagram} for the graphical illustration. Note that our framework allows $z_{ij} \neq z_{ij'}$ while $x_{ij} = x_{ij'}$ for $j \neq j'$. In other words, two measurements of the same order derivative at the same location may have different values due to measurement noise.

We introduce the notations:
\begin{equation}
  \begin{aligned}
    \bm z_0 &= \left[z_{01}, z_{02}, \cdots, z_{0p_0}\right]^\text{T} \\
    \bm z_1 &= \left[z_{11}, z_{12}, \cdots, z_{1p_1}\right]^\text{T} \\
            &\quad\vdots \\
    \bm z_n &= \left[z_{n1}, z_{n2}, \cdots, z_{np_n}\right]^\text{T}.
  \end{aligned}
\end{equation}
By Eqs.~\eqref{eq:prior}, \eqref{eq:decomp}, and \eqref{eq:var}, we have the multivariate Gaussian distribution:
\begin{equation}
  \begin{bmatrix}
    z_{q*} \\
    \bm z_0 \\
    \vdots \\
    \bm z_n
  \end{bmatrix}
  \thicksim
  \mathcal{N}
  \left(\begin{bmatrix}
    M_* \\
    M_0 \\
    \vdots \\
    M_n
  \end{bmatrix},
  \begin{bmatrix}
    K_{**} + \delta_q^2 & K_{*0} & \cdots & K_{*n} \\
    K_{0*} & K_{00} + \delta_0^2 I_{p_0} & \cdots & K_{0n} \\
    \vdots & \vdots & \ddots & \vdots \\
    K_{n*} & K_{n0} & \cdots & K_{nn} + \delta_n^2 I_{p_n}
  \end{bmatrix}
  \right),
\end{equation}
where $I_j$ represents the $j\times j$ identity matrix, and $M_*, M_i, K_{**}, K_{*i'}, K_{i*}, K_{ii'}$ are defined the same as in Eq.~\eqref{eq:prior}. Then, the posterior distribution of $z_{q*}$ is also a Gaussian distribution:
\begin{equation}
  \left(z_{q*}\; \big |\;\bm z_0, \cdots, \bm z_n \right) \thicksim \mathcal{N}(\mu', \sigma'^2),
\end{equation}
where
\begin{equation}
  \begin{aligned}
    \mu' &= M_* + \left[K_{*0}, \cdots, K_{*n}\right] \begin{bmatrix}
      K_{00} + \delta_0^2 I_{p_0} & \cdots & K_{0n} \\
      \vdots & \ddots & \vdots \\
      K_{n0} & \cdots & K_{nn} + \delta_n^2 I_{p_n}
    \end{bmatrix}^{-1}
    \begin{bmatrix}
      \bm z_0 - M_0 \\
      \vdots \\
      \bm z_n - M_n
    \end{bmatrix} \\
    \sigma'^2 &= K_{**} + \delta_q^2 - \left[K_{*0}, \cdots, K_{*n}\right]
    \begin{bmatrix}
      K_{00} + \delta_0^2 I_{p_0} & \cdots & K_{0n} \\
      \vdots & \ddots & \vdots \\
      K_{n0} & \cdots & K_{nn} + \delta_n^2 I_{p_n}
    \end{bmatrix}^{-1}
    \begin{bmatrix}
      K_{0*} \\
      \vdots \\
      K_{n*}
    \end{bmatrix}. \\
  \end{aligned}
\end{equation}
Similar to the noiseless scenario, the hyperparameters can be identified via maximizing the following log-likelihood:
\begin{multline}
  \log p\left(\bm z_0, \cdots, \bm z_n \;\big |\; \bm x_0, \cdots, \bm x_n, \bm \theta, \delta_0, \cdots, \delta_n\right) = -\frac{p_0+\cdots+p_n}{2}\log(2\pi) \\
  -\;\frac{1}{2} \log\left(\det\left(\begin{bmatrix}
      K_{00} + \delta_0^2 I_{p_0} & \cdots & K_{0n} \\
      \vdots & \ddots & \vdots \\
      K_{n0} & \cdots & K_{nn} + \delta_n^2 I_{p_n}
    \end{bmatrix}\right)\right) \\
  -\;\frac{1}{2} \begin{bmatrix}
    \bm z_0 - M_0 \\
    \vdots \\
    \bm z_n - M_n
  \end{bmatrix}^\text{T} \begin{bmatrix}
      K_{00} + \delta_0^2 I_{p_0} & \cdots & K_{0n} \\
      \vdots & \ddots & \vdots \\
      K_{n0} & \cdots & K_{nn} + \delta_n^2 I_{p_n}
    \end{bmatrix}^{-1} \begin{bmatrix}
    \bm z_0 - M_0 \\
    \vdots \\
    \bm z_n - M_n
  \end{bmatrix}.
\end{multline}
The difference is that we also identify the noise intensities $\delta_0,\cdots,\delta_n$ in the optimization.

\section{Numerical examples}
\label{sec:results}
In this section, we present four examples to illustrate the AGRF framework. In all examples, we use the zero mean function
\begin{equation}
  m(x) = 0
\end{equation}
and the squared exponential covariance function
\begin{equation}
  k(x,x') = a^2 \exp\left(\frac{-(x-x')^2}{2 l^2}\right)
\end{equation}
to construct the GRF representing the observable. The hyperparameters to identify are $\bm \theta = (a, l)$. We use the following relative $L_2$ error (RLE) to evaluate the accuracy of the prediction by different GRF-based methods:
\begin{equation}
    \text{RLE} = \frac{\Vert u-\tilde u\Vert_2}{\Vert u\Vert_2},
\end{equation}
where $u$ is the exact function and $\tilde u$ is the approximation.

\subsection{Composite function (noiseless)}
\begin{figure}[p]
  \centering
  \includegraphics[width=0.99\textwidth]{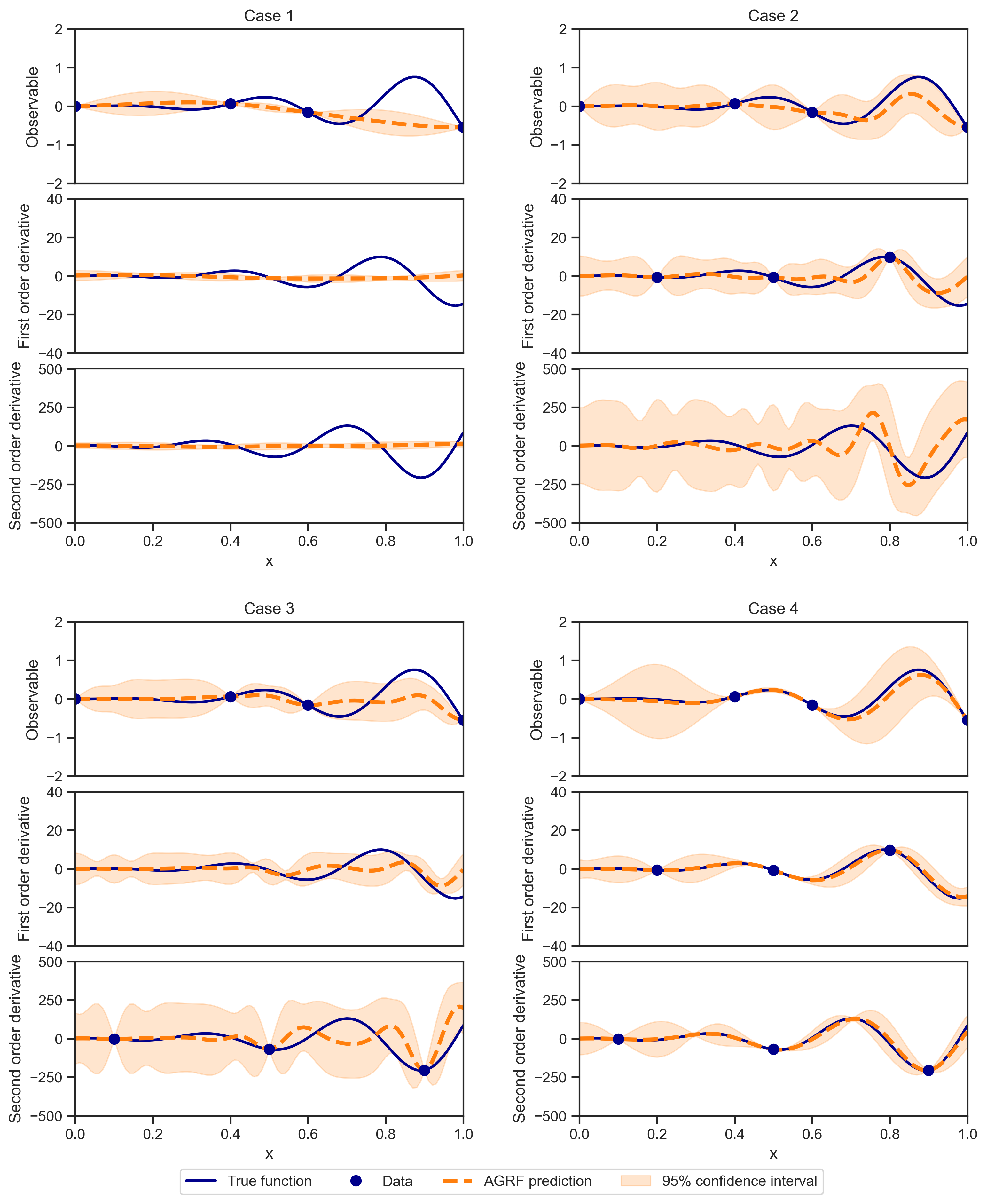}
  \caption{\textbf{[Composite function (noiseless)] Prediction of the observable, first order derivative, and second order derivative by AGRF.} Case~1: the data include the observable only. Case~2: the data include the observable and first order derivative. Case~3: the data include the observable and second order derivative. Case~4: the data include the observable, first order derivative, and second order derivative. AGRF is able to integrate the observable and derivatives of any order, regardless of the location where they are collected. The AGRF prediction improves when more information is available.}
  \label{fig:exp1_prediction}
\end{figure}

\begin{figure}[t]
  \centering
  \includegraphics[width=0.495\textwidth]{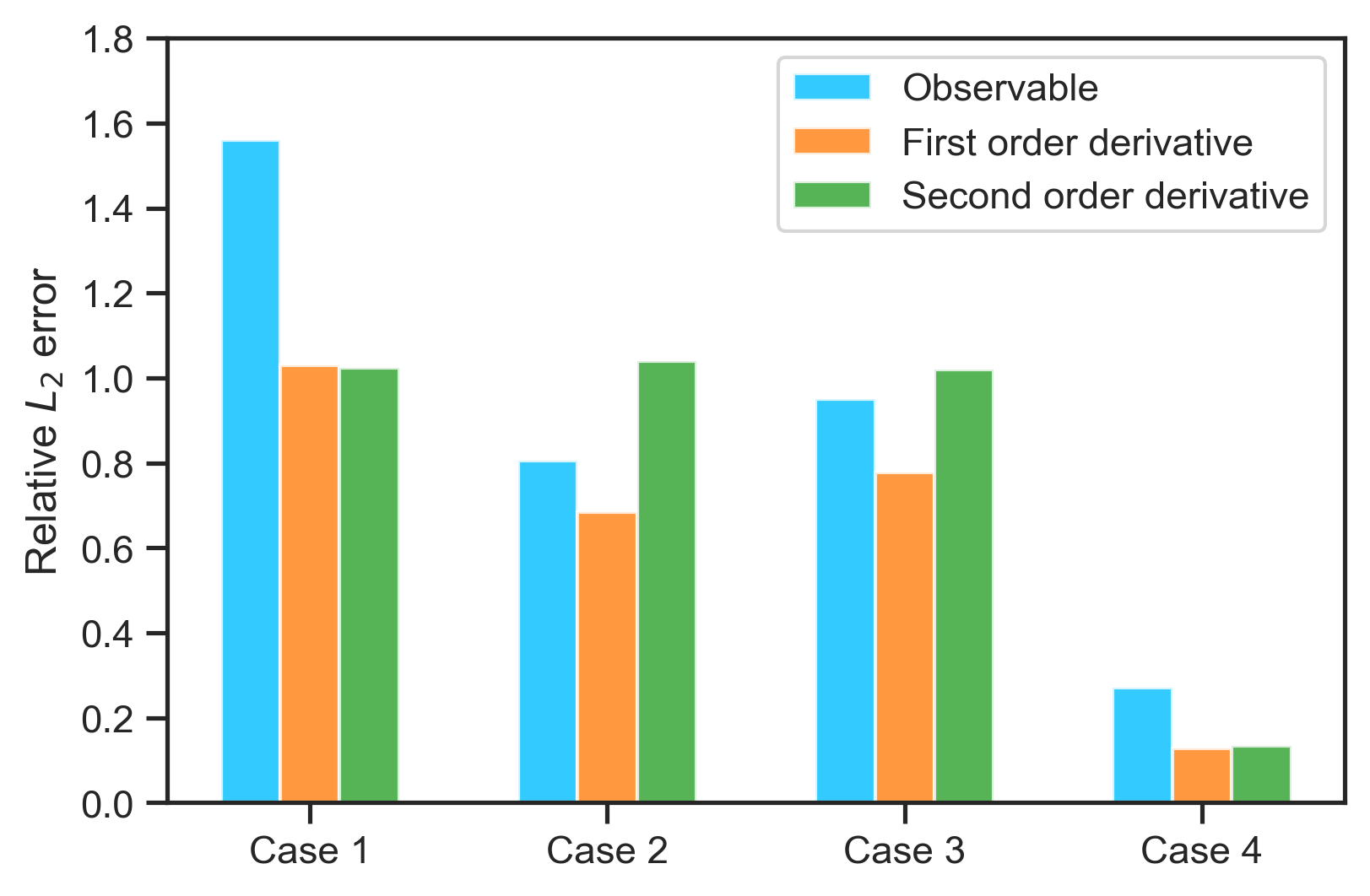}
  \caption{\textbf{[Composite function (noiseless)] Comparison of the prediction accuracy of AGRF in different cases.} See Figure~\ref{fig:exp1_prediction} for more explanations.}
  \label{fig:exp1_RLE}
\end{figure}

Consider the following function:
\begin{equation}
  y(x) = x^2 \sin (16x-6)
\end{equation}
on $x\in [0,1]$. The available data may include the observable at $x\in \{0.0, 0.4, 0.6, 1.0\}$, its first order derivative at $x\in \{0.2, 0.5, 0.8\}$, and its second order derivative at $x\in \{0.1, 0.5, 0.9\}$. We consider the following cases:
\begin{description}
  \item[Case~1:] the data include the observable only.
  \item[Case~2:] the data include the observable and first order derivative.
  \item[Case~3:] the data include the observable and second order derivative.
  \item[Case~4:] the data include the observable, first order derivative, and second order derivative.
\end{description}
See Figure~\ref{fig:exp1_prediction} for the prediction by AGRF. The observable, first order derivative, and second order derivative are predicted in each case. In Case~1, the AGRF model for the observable is the same as the conventional GP model, because the former is a generalization of the latter. However, the conventional GP regression does not provide the prediction of derivatives. In Case~2, the AGRF prediction matches the true function and its derivatives better than Case~1. This is because the derivative information is incorporated in the model. Similarly, in Case~3, the prediction is enhanced by incorporating the second order derivative, and thus is better than Case~1. In Case~4, AGRF has the best prediction among all four cases. It is not surprising that by using all the available information we can construct the most accurate surrogate model. See Figure~\ref{fig:exp1_RLE} for a quantitative comparison of the prediction accuracy.

In this example, we can see that AGRF is able to integrate the observable and derivatives of any order, regardless of the location where they are collected. As one might expect, the AGRF prediction improves when more information is available.

\subsection{Damped harmonic oscillator (noiseless)}
\begin{figure}[p]
    \centering
    \includegraphics[width=0.99\textwidth]{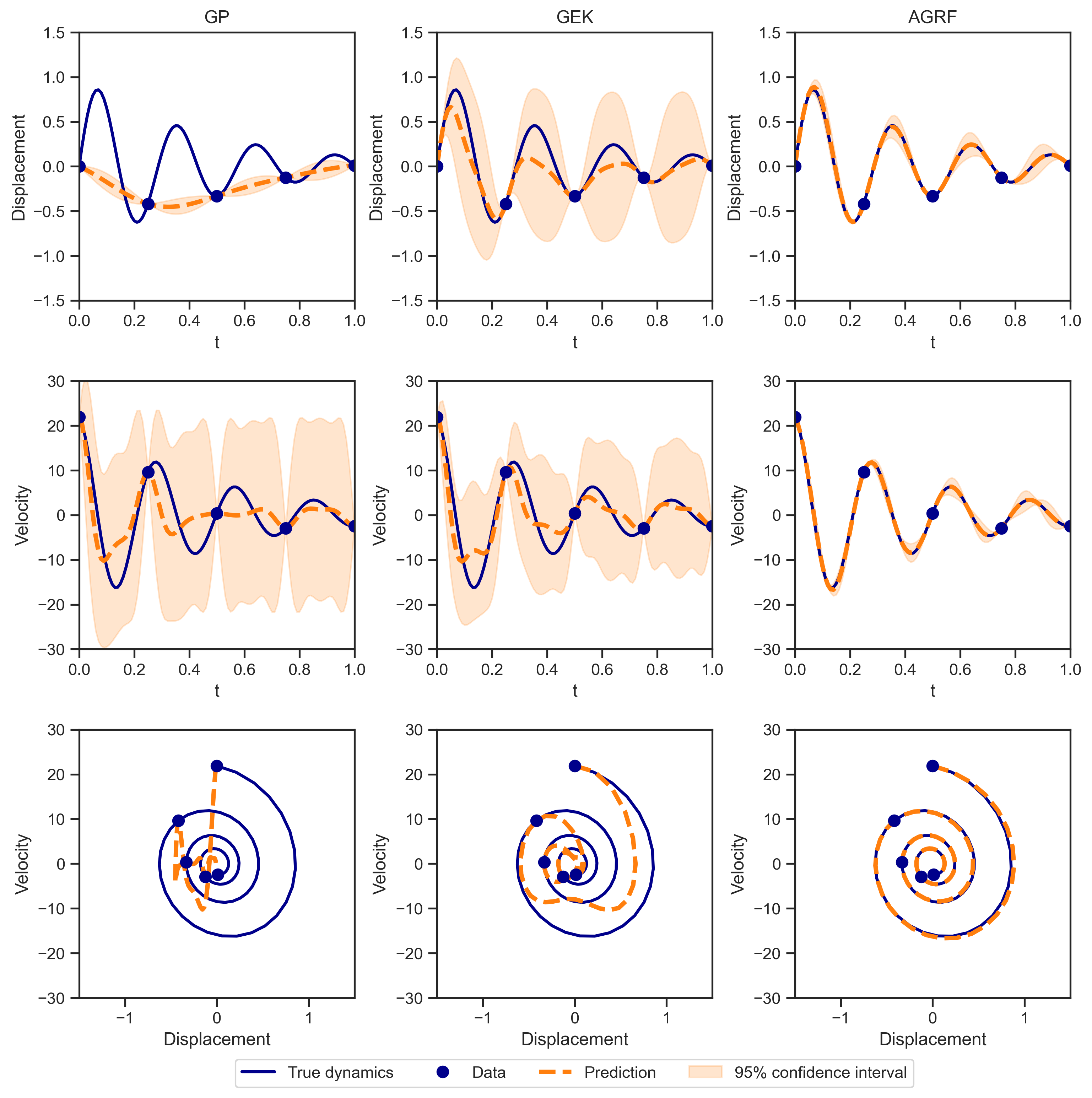}
    \caption{\textbf{[Damped harmonic oscillator (noiseless)] Prediction of the displacement, velocity, and phase-space diagram by different methods.} GP: the data include the observable and first order derivative; the observable data are used to predict the displacement and the first order derivative data are used to predict the velocity, respectively. GEK: the data include the observable and first order derivative; all the data are used jointly in the same random field to predict the displacement and velocity at the same time. AGRF: the data include the observable, first order derivative, and second order derivative; all the data are used together in the same random field to predict the displacement and velocity at the same time. GEK produces better prediction than GP, while AGRF predicts more accurately than GEK. By using all the available information together in the same random field, we can construct the most accurate surrogate model.}
    \label{fig:exp2_prediction}
\end{figure}

\begin{figure}[t]
  \centering
  \includegraphics[width=0.495\textwidth]{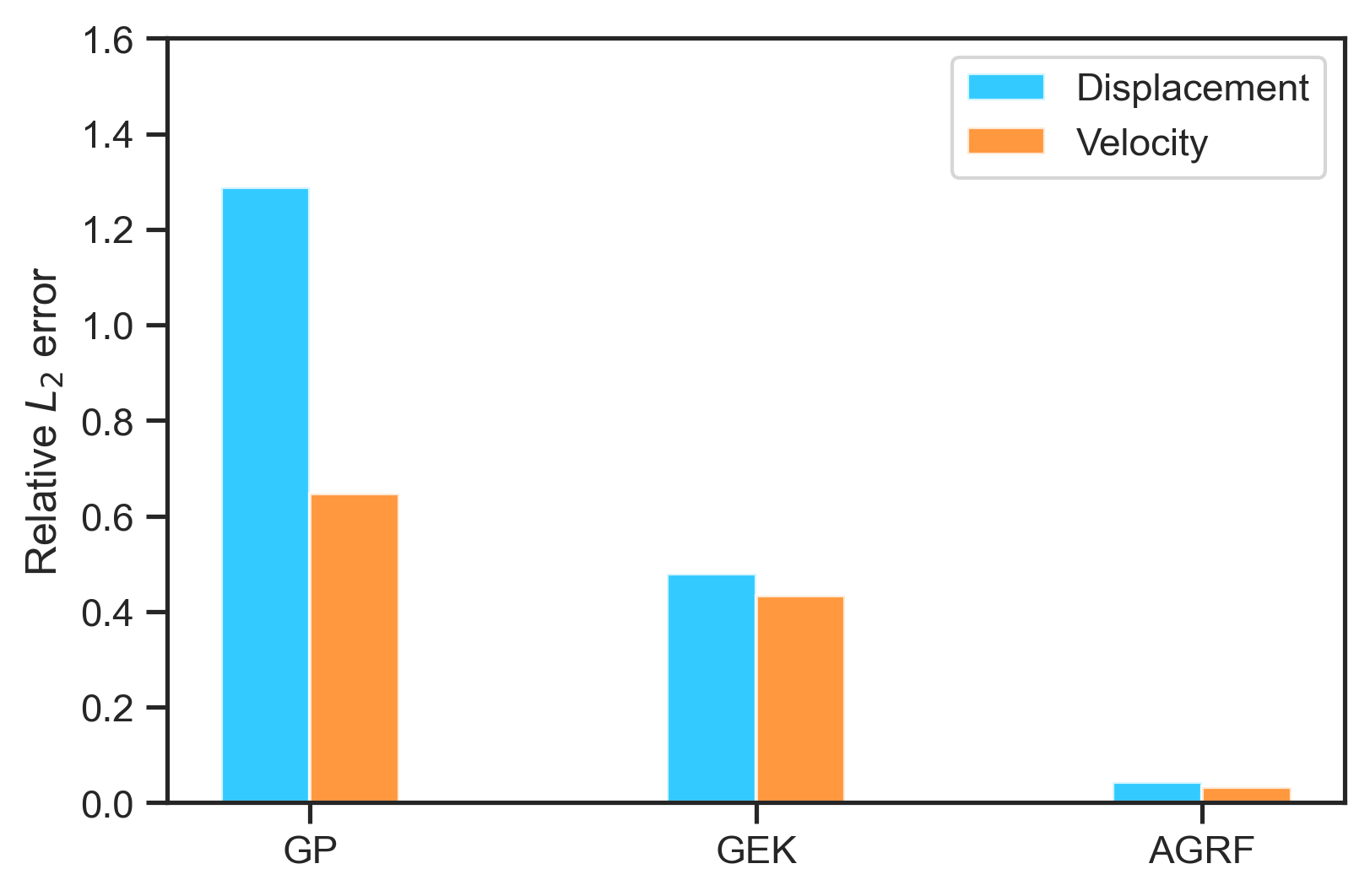}
  \caption{\textbf{[Damped harmonic oscillator (noiseless)] Comparison of the prediction accuracy by different methods.} See Figure~\ref{fig:exp2_prediction} for more explanations.}
  \label{fig:exp2_RLE}
\end{figure}

Consider the following damped harmonic oscillator:
\begin{equation}
  \left\{
  \begin{aligned}
    F &= -ky-cy' \\
    F &= my''
  \end{aligned}
  \right.
\end{equation}
where $y$ is the displacement, $y'$ is the velocity, $y''$ is the acceleration, $-ky$ is the restoring force, and $-cy'$ is the frictional force. This system can be simplified to:
\begin{equation}
  y'' + 2\zeta \omega_0 y' + \omega_0^2 y = 0,
\end{equation}
where $\omega_0 = \sqrt{k/m}$ is the undamped angular frequency and $\zeta = c/\sqrt{4mk}$ is the damping ratio. When $\zeta < 1$, it has the solution:
\begin{equation}
  y(t) = A \exp(-\zeta \omega_0 t) \sin\left(\sqrt{1-\zeta^2}\omega_0 t + \phi\right),
\end{equation}
where the amplitude $A$ and the phase $\phi$ determine the behavior needed to match the initial conditions. Now, consider a specific example:
\begin{equation}
  y(t) = \exp(-0.1\cdot 22\cdot t) \sin\left(\sqrt{1-0.1^2}\cdot 22\cdot t\right)
\end{equation}
on $t\in [0,1]$. The available data may include the observable, its first order derivative, and its second order derivative at $x\in \{0.0, 0.25, 0.5, 0.75, 1.0\}$. We use conventional Gaussian process regression (GP, the same as AGRF using the data of observable only), gradient-enhanced Kriging~\cite{laurent_overview_2019} (GEK, the same as AGRF using the data of observable and first order derivative at the same location), or AGRF (using the data of observable and all derivatives available) to construct the surrogate model:
\begin{description}
  \item[GP:] the data include the observable and first order derivative; the observable data are used to predict the displacement and the first order derivative data are used to predict the velocity, respectively.
  \item[GEK:] the data include the observable and first order derivative; all the data are used jointly in the same random field to predict the displacement and velocity at the same time.
  \item[AGRF:] the data include the observable, first order derivative, and second order derivative; all the data are used together in the same random field to predict the displacement and velocity at the same time.
\end{description}
See Figure~\ref{fig:exp2_prediction} for the prediction by different methods. The displacement, velocity, and phase-space diagram are predicted by each method. GEK produces better prediction than GP, while AGRF predicts more accurately than GEK. By using all the available information together in the same random field, we can construct the most accurate surrogate model. See Figure~\ref{fig:exp2_RLE} for a quantitative comparison of the prediction accuracy.

\subsection{Korteweg-De Vries equation (noisy)}
\begin{figure}[p]
    \centering
    \includegraphics[width=0.99\textwidth]{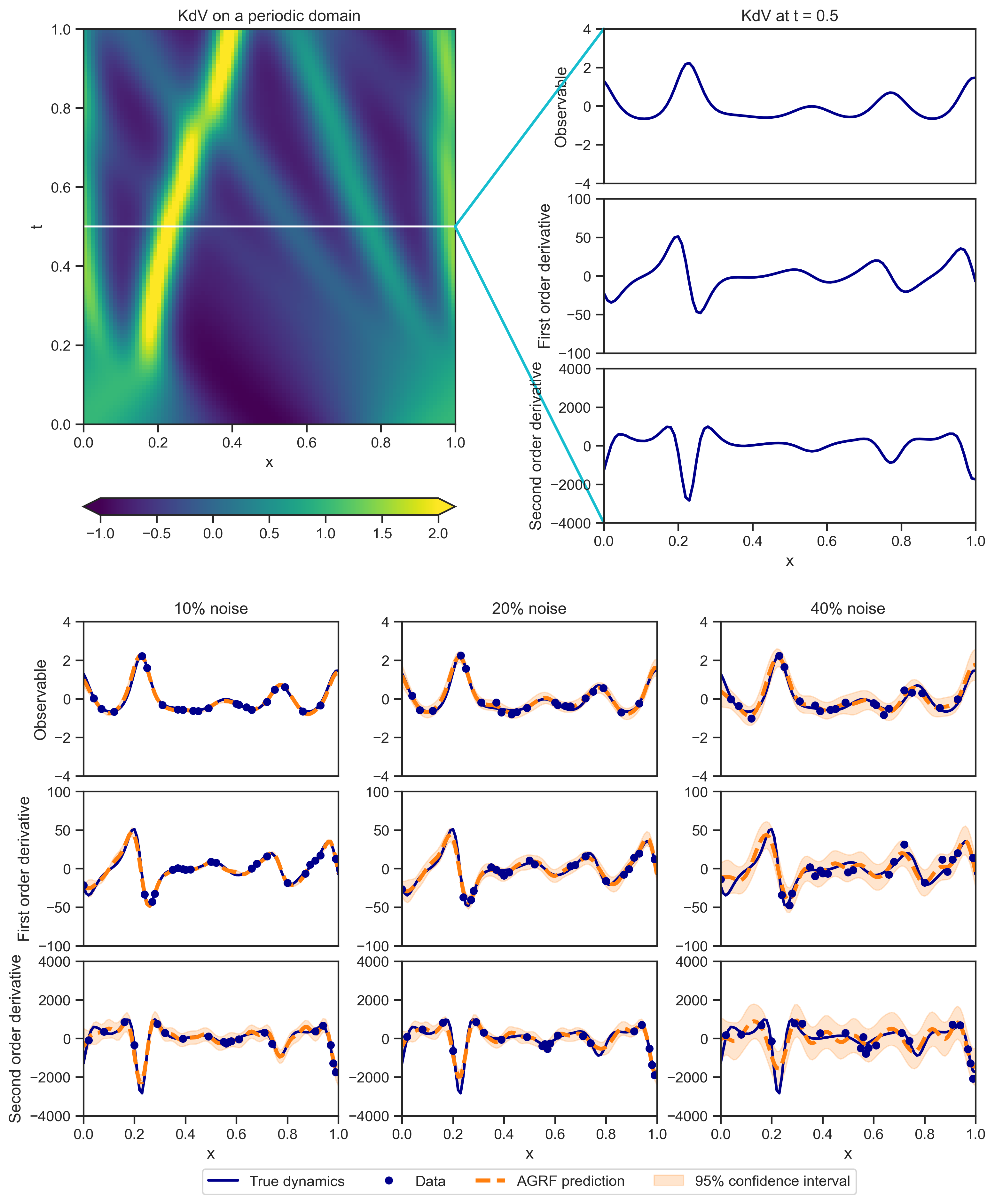}
    \caption{\textbf{[Korteweg-De Vries equation (noisy)] Top: the solution at $t=0.5$ is studied. Bottom: prediction of the observable, first order derivative, and second order derivative by AGRF under different levels of noise.} AGRF has good performance even when the noise is as high as 40\%. As one might expect, the AGRF prediction is better when the noise is lower.}
    \label{fig:exp3_prediction}
\end{figure}

\begin{figure}[t]
  \centering
  \includegraphics[width=0.495\textwidth]{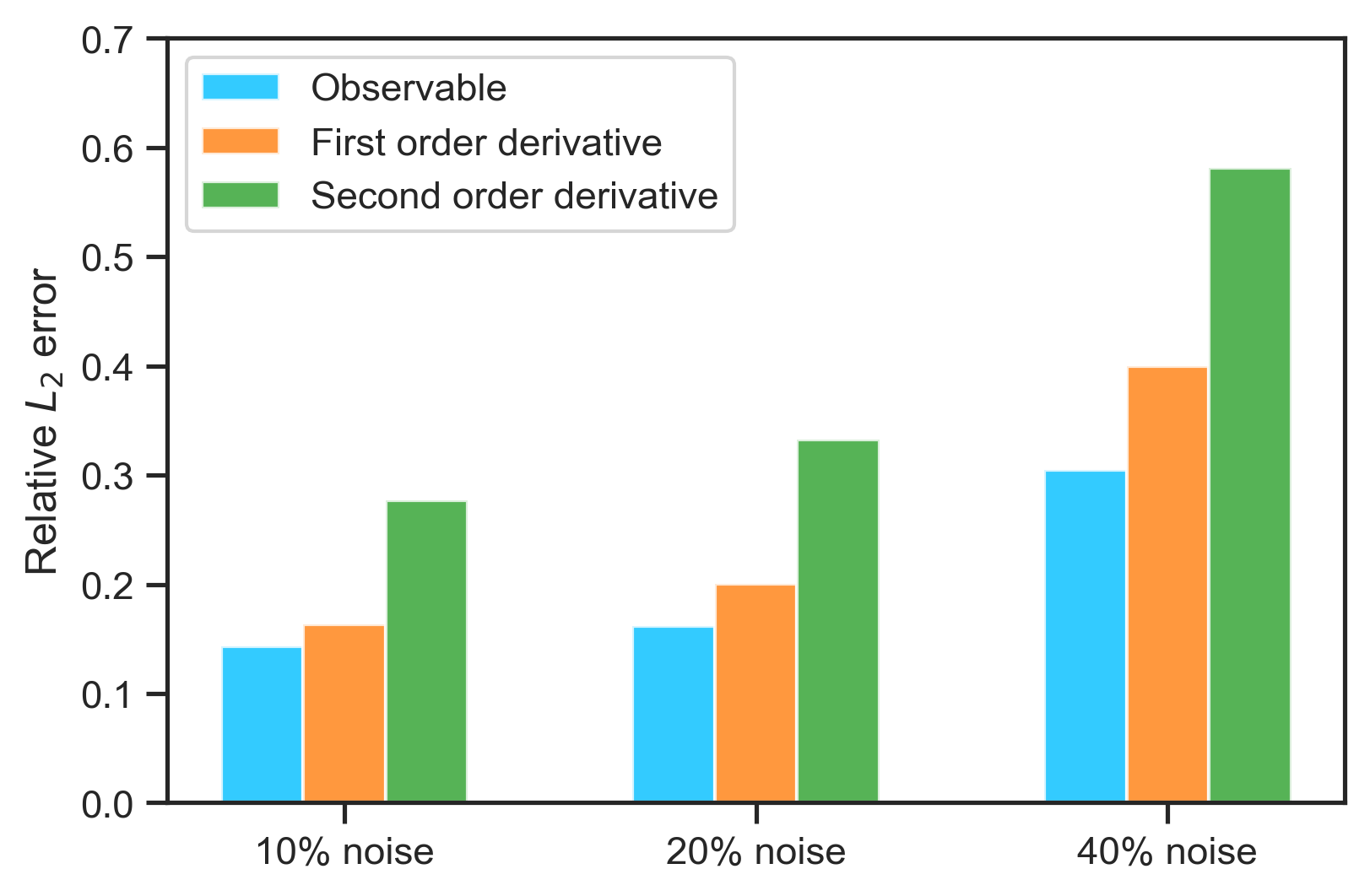}
  \caption{\textbf{[Korteweg-De Vries equation (noisy)] Comparison of the prediction accuracy under different levels of noise.} See Figure~\ref{fig:exp3_prediction} for more explanations.}
  \label{fig:exp3_RLE}
\end{figure}

Consider the following Korteweg-De Vries (KdV) equation:
\begin{equation}
    \frac{\partial u}{\partial t} + u\frac{\partial u}{\partial x} + 0.0005\frac{\partial^3 u}{\partial x^3} = 0
\end{equation}
on $x \in [0,1]$ and $t \in [0,\infty)$ with an initial condition $u(x,0)=\cos (2\pi x)$ and periodic boundary condition. The solution at $t=0.5$ is studied. We try to approximate the observable, its first order derivative, and its second order derivative. The data are randomly chosen from the observable, first order derivative, and second order derivative; 20 from each, 60 in total. Then, noise is added to the data. We study the performance of AGRF under different levels of noise: 10\%~noise, 20\%~noise, and 40\%~noise. See Figure~\ref{fig:exp3_prediction} for the prediction. As one might expect, the AGRF prediction is better when the noise is lower. See Figure~\ref{fig:exp3_RLE} for a quantitative comparison of the prediction accuracy.

\subsection{Burgers' equation (noisy)}
\begin{figure}[p]
    \centering
    \includegraphics[width=0.99\textwidth]{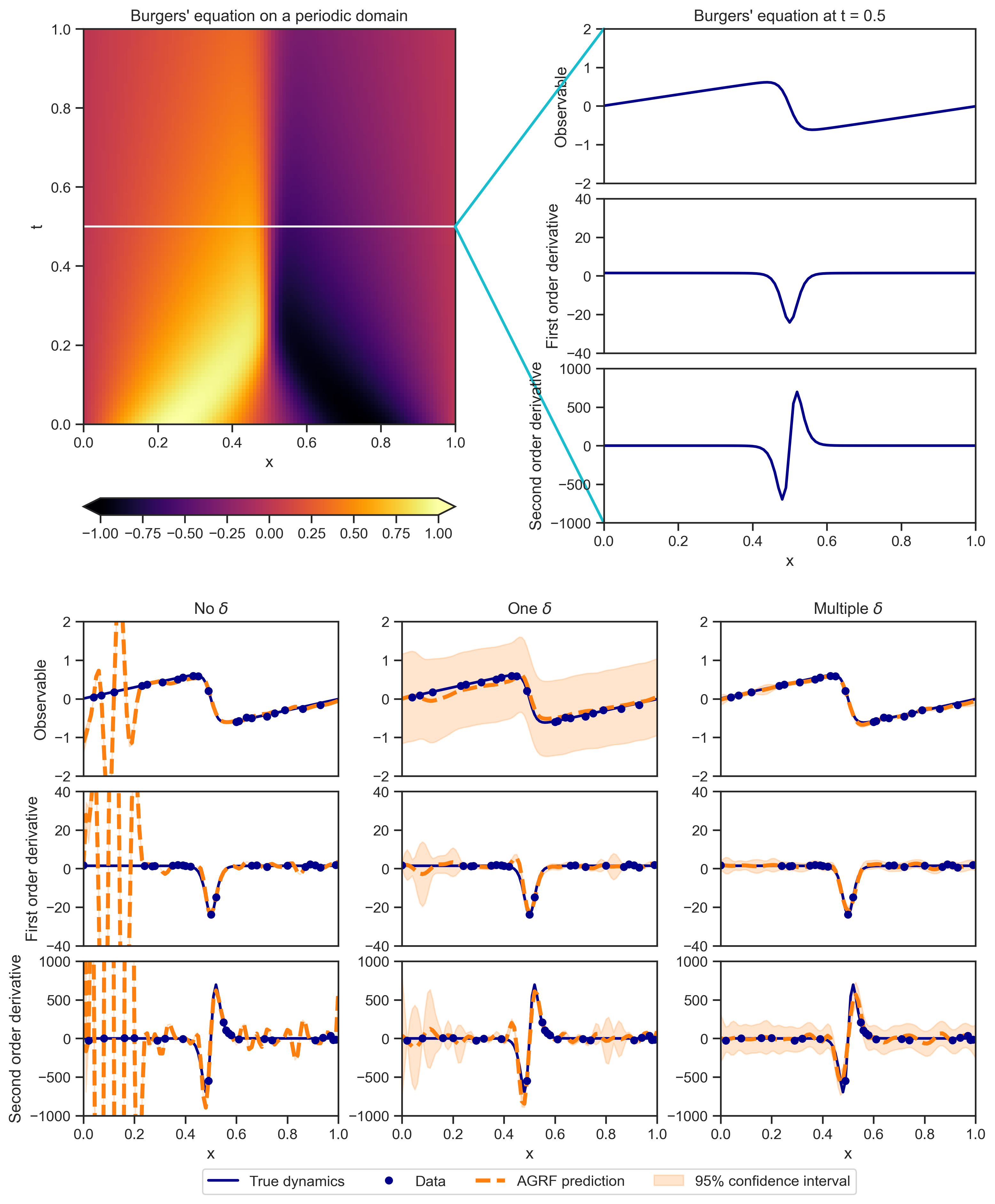}
    \caption{\textbf{[Burgers' equation (noisy)] Top: the solution at $t=0.5$ is studied. Bottom: prediction of the observable, first order derivative, and second order derivative by different AGRF calibrations.} No~$\delta$: noiseless formulation is used despite the presence of noise in the data, i.e., $\delta_0 = \delta_1 = \delta_2 = 0$ in Eqn.~\eqref{eq:decomp}. One~$\delta$: the same noise intensity is used for different order derivatives, i.e., $\delta_0 = \delta_1 = \delta_2$ in Eqn.~\eqref{eq:decomp}. Multiple~$\delta$: different noise intensities are used for different order derivatives, i.e., the same as Eqn.~\eqref{eq:decomp}. When the noiseless formulation is used despite the presence of noise in the data, overfitting is an issue. When the same noise intensity is used for different order derivatives, the uncertainty in the prediction is incompatible with the data since different order derivatives have different scales. When the formulation is exactly the same as Eqn.~\eqref{eq:decomp}, AGRF has the best performance.}
    \label{fig:exp4_prediction}
\end{figure}

\begin{figure}[t]
  \centering
  \includegraphics[width=0.495\textwidth]{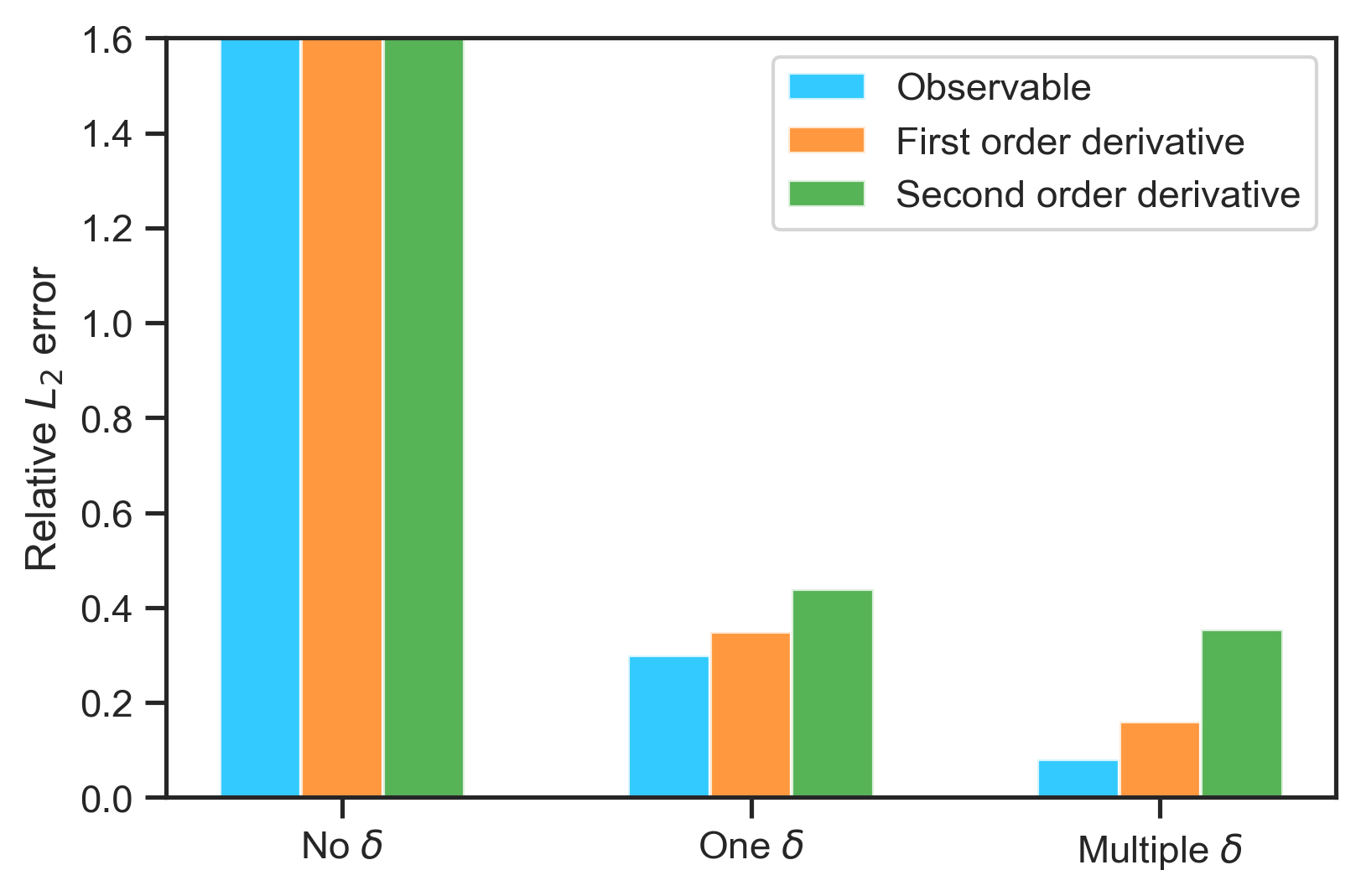}
  \caption{\textbf{[Burgers' equation (noisy)] Comparison of the prediction accuracy by different AGRF calibrations.} See Figure~\ref{fig:exp4_prediction} for more explanations. The relative $L_2$ errors in the case ``no $\delta$'' are greater than $1.6$ and out of bound.}
  \label{fig:exp4_RLE}
\end{figure}

Consider the following Burgers' equation:
\begin{equation}
    \frac{\partial u}{\partial t} + u\frac{\partial u}{\partial x} - 0.01\frac{\partial^2 u}{\partial x^2} = 0
\end{equation}
on $x \in [0,1]$ and $t \in [0,\infty)$ with an initial condition $u(x,0)=\sin (2\pi x)$ and periodic boundary condition. The solution at $t=0.5$ is studied. We try to approximate the observable, its first order derivative, and its second order derivative. The data are randomly chosen from the observable, first order derivative, and second order derivative; 20 from each, 60 in total. Then, 10\%~noise is added to the data. We study the performance of AGRF when it is calibrated in different ways:
\begin{description}
  \item[No~$\bm \delta$:] noiseless formulation is used despite the presence of noise in the data, i.e., $\delta_0 = \delta_1 = \delta_2 = 0$ in Eqn.~\eqref{eq:decomp}.
  \item[One~$\bm \delta$:] the same noise intensity is used for different order derivatives, i.e., $\delta_0 = \delta_1 = \delta_2$ in Eqn.~\eqref{eq:decomp}.
  \item[Multiple~$\bm \delta$:] different noise intensities are used for different order derivatives, i.e., the same as Eqn.~\eqref{eq:decomp}.
\end{description}
See Figure~\ref{fig:exp4_prediction} for the prediction by different calibrations. When the noiseless formulation is used despite the presence of noise in the data, overfitting is an issue. When the same noise intensity is used for different order derivatives, the uncertainty in the prediction is incompatible with the data since different order derivatives have different scales. When the formulation is exactly the same as Eqn.~\eqref{eq:decomp}, AGRF has the best performance. See Figure~\ref{fig:exp4_RLE} for a quantitative comparison of the prediction accuracy.

\section{Conclusion}
\label{sec:concl}
In this work, we propose the novel augmented Gaussian random field (AGRF), which is a universal framework incorporating the observable and its derivatives of any order. A comprehensive theoretical foundation is laid. We prove that given the smoothness of the mean and covariance functions, the observable and all its derivatives are governed by a single GRF, which is the aforementioned AGRF. As a corollary, the intuitive yet subtle statement ``the derivative of a Gaussian process remains a Gaussian process'', which is widely used in probabilistic scientific computing and GP-based regression methods, is validated.

Furthermore, the computational method corresponding to the universal AGRF framework is constructed. Both noiseless and noisy scenarios are considered. Formulas of the posterior distributions are deduced in a nice closed form. In the noisy scenario, we use different noise intensities for different order derivatives because different order derivatives might have different magnitudes. We provide four numerical examples to demonstrate that: (1) AGRF is able to integrate the observable and derivatives of any order, regardless of the location where they are collected. The AGRF prediction improves when more information is available. (2) By using all the available information together in the same random field, we can construct the most accurate surrogate model. (3) AGRF has good performance even when the noise is as high as 40\%. The AGRF prediction is better when the noise is lower. (4) When the noiseless formulation is used despite the presence of noise in the data, overfitting is an issue. When the same noise intensity is used for different order derivatives, the uncertainty in the prediction is incompatible with the data since different order derivatives have different scales. When the formulation is exactly the same as described in this paper, AGRF has the best performance.

A significant advantage of our computational method is that the universal AGRF framework provides a natural way to incorporate arbitrary order derivatives and deal with missing data. New research directions and applications may be opened up following this universal framework. Although one-dimensional systems are demonstrated in this paper, our conclusion can be extended to multi-dimensional systems.

Similar to the conventional Gaussian process regression, the bottleneck of AGRF is that the computational formula involves solving linear systems and calculating determinant, both of which have cubic complexity to data size. Approximation methods for Gaussian process regression may be adapted to the AGRF framework to improve its scalability. A review of such methods is presented in \cite{liu_gaussian_2020}.

\section*{Acknowledgements}
SZ was supported by National Science Foundation (NSF) Mathematical Sciences Graduate Internship Program sponsored by NSF Division of Mathematical Sciences. XY was supported by U.S. Department of Energy, Office of Science, Advanced Scientific Computing Research as part of Multifaceted Mathematics for Rare, Extreme Events in Complex Energy and Environment Systems. GL was supported by National Science Foundation (DMS-1555072, DMS-1736364, DMS-2053746, CMMI-1634832, and CMMI-1560834), Brookhaven National Laboratory Subcontract (382247), and U.S. Department of Energy Office of Science Advanced Scientific Computing Research (DE-SC0021142).

\bibliographystyle{ieeetr}
\bibliography{ref}

\end{document}